\documentclass[pdflatex,sn-mathphys-num]{sn-jnl}

\usepackage{graphicx}%
\usepackage{multirow}%
\usepackage{amsmath,amssymb,amsfonts}%
\usepackage{amsthm}%
\usepackage{mathrsfs}%
\usepackage[title]{appendix}%
\usepackage{xcolor}%
\usepackage{textcomp}%
\usepackage{manyfoot}%
\usepackage{booktabs}%
\usepackage{algorithm}%
\usepackage{algorithmicx}%
\usepackage{algpseudocode}%
\usepackage{listings}%

\usepackage{physics}
\usepackage{mathtools}
\mathtoolsset{showonlyrefs}
\usepackage{bm}

\theoremstyle{thmstyleone}%
\newtheorem{thm}{Theorem}[section]
\newtheorem{lemma}[thm]{Lemma}

\newtheorem{cor}[thm]{Corollary}
\newtheorem{remark}[thm]{Remark}%

\theoremstyle{thmstyletwo}%

\begin{document}

\title[Article Title]{Efficient LU factorization exploiting direct-indirect Burton--Miller equation for Helmholtz transmission problems}


\author*[1]{\fnm{Yasuhiro} \sur{Matsumoto}}\email{matsumoto@cii.isct.ac.jp}

\author[2]{\fnm{Kei} \sur{Matsushima}}\email{matsushima@acs.hiroshima-u.ac.jp}

\affil*[1]{\orgdiv{Center for Information Infrastructure}, \orgname{Institute of Science Tokyo}, \orgaddress{\street{2-12-1-I8-21 Ookayama}, \city{Meguro-ku}, \postcode{152-8550}, \state{Tokyo}, \country{Japan}}}

\affil*[2]{\orgdiv{Graduate School of Advanced Science and Engineering}, \orgname{Hiroshima University}, \orgaddress{\street{1-7-1 Kagamiyama}, \city{Higashi Hiroshima}, \postcode{739-8521}, \state{Hiroshima}, \country{Japan}}}

%
%
%
%
%


\abstract{
This paper proposes a direct-indirect mixed Burton--Miller boundary integral equation for solving Helmholtz scattering problems with transmissive scatterers.
The proposed formulation has three unknowns, one more than the number of unknowns for the ordinary formulation.
However, we can construct efficient numerical solvers based on LU factorization by exploiting the sparse alignment of the boundary integral operators of the proposed formulation.
Numerical examples demonstrate that the direct solver based on the proposed formulation is approximately 40\% faster than the ordinary formulation when the LU-factorization-based solver is used.
In addition, the proposed formulation is applied to a fast direct solver employing LU factorization in its algorithm.
In the application to the fast direct solver, the proxy method with a weak admissibility low-rank approximation is developed.
The speedup achieved using the proposed formulation is also shown to be effective in finding nonlinear eigenvalues, which are related to the uniqueness of the solution, in boundary value problems.
Furthermore, the well-posedness of the proposed boundary integral equation is established for scatterers with boundaries of class $C^2$, using the mapping property of boundary integral operators in H\"older space.
}




\maketitle

\section{Introduction}
The wave scattering problem is an important mathematical model in physics and engineering.
Peculiar phenomena exhibited by non-Hermitian systems with respect to the wave scattering problem
have recently attracted attention \cite{ashida2020non, zhang2022review}.
These phenomena are related to the nonlinear eigenvalue problem for boundary value problems.
Studies on acoustics have investigated exceptional points \cite{matsushima2025exceptional}
and skin effects \cite{huang2024acoustic}.

Nonlinear eigenvalue problems can be handled using methods such as the Sakurai--Sugiura projection method \cite{asakura2009numerical}
and the adaptive Antoulas--Anderson algorithm \cite{nakatsukasa2018aaa},
which involve repeatedly solving boundary value problems with respect to complex parameters (e.g., complex frequencies).
To solve nonlinear eigenvalue problems for boundary value problems,
a fast and accurate numerical method is required
since an analytical solution is available only in a special case.
Furthermore, for these problems,
it is necessary to distinguish the computed eigenvalues that arise from the numerical method and are therefore unrelated to the original boundary value problem.
For example, when analyzing scattering problems that include an unbounded region,
domain-based methods such as the finite element method
require some absorbing boundary condition to deal with radiation conditions,
such as the perfectly matched layer \cite{berenger1994perfectly},
to approximate the unbounded region with a bounded computational region.
However, truncating the domain with the perfectly matched layer introduces fictitious eigenvalues;
therefore, care must be taken when handling this approximation \cite{ARAUJOC2021110024, kim2009computation}.
In contrast, numerical methods based on boundary integral equations do not need to truncate the computational domain, owing to the property of the fundamental solutions that satisfy the radiation condition.
Although these methods also suffer from fictitious eigenvalues stemming from their formulation, they can distinguish between true eigenvalues (i.e., those that originate from the boundary value problem itself) and fictitious eigenvalues simply by switching the fundamental solution to the inward solution in transmission problems \cite{misawa2017boundary}.
When analyzing Neumann or Dirichlet problems, we can distinguish eigenvalues that originate from the boundary value problem itself
by selecting a constant parameter in the boundary integral equations based on the Burton--Miller method \cite{ZHENG201543, burton1971application}.
The volume integral equation, namely the Lippmann--Schwinger equation, can be used to analyze scattering problems since it is free from fictitious eigenvalues \cite{ARAUJOC2021110024}.
However, its computational load is quite large.
In the view of the high-order method described below,
the domain-based spectral method \cite{gottlieb1977numerical} is attractive.
However, this method needs some absorbing boundary conditions to truncate an unbounded region to a bounded one in a scattering problem \cite{shapoval2019two}.
Thus, we are interested in numerical methods based on boundary integral equations.
Here, we refer to numerical methods based on boundary integral equations as
boundary integral equation methods.

Some high-order discretization methods and fast methods are available for boundary integral equation methods.
Since a coefficient matrix of discretized systems is dense in a boundary integral equation method,
the availability of these techniques is important.
High-order discretization methods have been applied to Laplace boundary value problems \cite{beale2024adjoint}, elastic wave scattering \cite{lai2022fast}, Maxwell's equation \cite{ganesh2025all}, and biharmonic wave scattering \cite{dong2024novel}.
They can obtain a numerical solution that converges at an algebraic rate of $O(h^{p})$, where $h$ is a representative length of the mesh and $p$ is some positive integer, making them fast.
Moreover, high-order discretization methods and fast methods such as the fast multipole method \cite{rokhlin1985rapid}
can be used simultaneously in some cases.
Fast methods can be broadly divided into two types: fast iterative solvers and fast direct solvers.
Fast iterative solvers, such as the fast multipole method \cite{rokhlin1985rapid} and its variants (e.g., \cite{schanz2018fast, YESYPENKO2025113707}),
rely mainly on the fast matrix-vector multiplication in the Krylov subspace method \cite{sogabe2022krylov}.
The $\mathcal{H}$-matrix method \cite{borm2003introduction} has made it possible to apply fast direct solvers, whose computational cost is almost independent of the condition number of the coefficient matrix of a discrete system.
Note that while the matrix-vector multiplication is a level 2 BLAS operation,
techniques employed in the fast direct solver such as LU factorization or some other factorization scheme
are categorized as level 3 BLAS.
In recent years, accelerators such as graphics processing units have shown excellent performance for algorithms with high computational intensity, such as those with level 3 BLAS operations.
This is particularly true for matrix-matrix multiplication, which is often used in matrix factorization techniques.
Furthermore, dedicated circuits for the operations of matrix-matrix multiplication are increasingly being implemented, making direct solvers attractive.

This study proposes direct solvers for Helmholtz transmission problems based on the direct-indirect mixed Burton--Miller boundary integral equation.
The proposed boundary integral equation is a rearrangement of a previously proposed equation \cite{YasuhiroMATSUMOTO202308-231124} that uses both direct and indirect integral representations.
With this rearrangement, the proposed equation has a sparse structure of the boundary integral operators, which enables the construction of efficient LU factorization for the discretized equation.
This efficient direct solver technique can be extended to fast direct solvers that use standard LU factorization in their algorithm.
Using the proposed efficient LU factorization, we formulate a variant of the Martinsson--Rokhlin fast direct solver \cite{MARTINSSON20051}, which is a recursive compression technique for a discretized boundary integral equation that uses shared coefficients based on a weak admissibility condition.
If a size $N$ of a coefficient matrix of the discrete system is small,
the proposed (non-fast) direct solver, which has $O(N^3)$ time complexity, may outperform the proposed fast direct solver, which has $O(N)$ time complexity, since the constant coefficient of the time complexity of the latter is large.
The two proposed direct solvers can handle problems ranging from small to large scales and are expected to be more efficient than the corresponding standard direct solvers for the ordinary Burton--Miller boundary integral equation for Helmholtz transmission problems.
The performance of the proposed direct and fast direct solvers is demonstrated using several numerical examples.
Moreover, the applicability of these solvers to nonlinear eigenvalue problems in boundary value problems using the Sakurai--Sugiura projection method is demonstrated using a high-order Nystr\"om method.
In addition to the development of the direct solvers, we establish the well-posedness of the direct-indirect mixed Burton--Miller boundary integral equation in the case of $C^2$ boundaries, which extends a existing result on
the injectivity of its integral operator for the case where the boundary shape is a circle \cite{matsumoto2025injectivity}.
A formulation that uses both direct and indirect integral representations is also found in \cite{rapun2008mixed}; however, the boundary integral equation in the previous study, but not the proposed equation, has (fictitious) eigenvalues on the real axis.

The rest of this paper is structured as follows.
Section \ref{sec:Preliminaries} formulates the Helmholtz transmission problem and reviews the mapping property of the corresponding boundary integral operators.
Section \ref{sec:well-posed} proposes the direct-indirect mixed Burton--Miller boundary integral equation and proves the well-posedness of this equation in the framework of H\"older space.
Section \ref{sec:solvers} presents the details of the proposed direct solvers.
Section \ref{sec:numerical_examples} demonstrates the performance of the proposed solvers.
Finally, Section \ref{sec:conclusion} gives the conclusions.

\section{Preliminaries} \label{sec:Preliminaries}
\subsection{The Helmholtz transmission problem}
Let $\Omega_1$ be an open and bounded subset of $\mathbb R^d$ for $d = 2, 3$.
We assume that the exterior $\Omega_0 := \mathbb R^d \setminus {\overline \Omega_1}$ is connected.
Let $\Gamma:=\partial \Omega_1$ be the interface between the two regions.
We assume that $\Gamma$ is of class $C^2$.
Let $\varepsilon_i > 0$, $\omega > 0$, and $k_i := \omega \sqrt{\varepsilon_i} > 0$ be constants,
representing the material constant, angular frequency, and wavenumber, respectively, in each $\Omega_{i}$ ($i=0,1$).
Let $u^{\mathrm{in}}$ be the incident wave field
that is
a solution of the Helmholtz equation with wavenumber $k_0$
in $\mathbb{R}^d$.
Let us consider the following transmission problem:
find the scattered fields $u_0^{\mathrm{sc}} \in C^2(\Omega_0) \cap C^1({\overline \Omega_0})$ and $u_1^{\mathrm{sc}} \in C^2(\Omega_1) \cap C^1({\overline \Omega_1})$ such that
\begin{align}
  \begin{dcases}
    \varDelta u_0^{\mathrm{sc}} + k_0^2 u_0^{\mathrm{sc}} = 0 & \text{in }\Omega_0,
    \\
    \varDelta u_1^{\mathrm{sc}} + k_1^2 u_1^{\mathrm{sc}} = 0 & \text{in }\Omega_1,
    \\
    u_0^{\mathrm{sc}} + u^{\mathrm{in}} = u_1^{\mathrm{sc}} \quad (=: u),
    & \text{on } \Gamma,
    \\
    \frac{1}{\varepsilon_0} \nu \cdot \nabla \qty(u_0^{\mathrm{sc}} + u_0^{\mathrm{in}})
    = \frac{1}{\varepsilon_1} \nu \cdot \nabla u_1^{\mathrm{sc}} \quad (=: q)
    & \text{on } \Gamma,
    \\
    {|x|}^{\frac{d-1}{2}}\left( \frac{\partial u_0^{\mathrm{sc}}}{\partial |x|} - ik u_0^{\mathrm{sc}} \right) \to 0 & \text{uniformly as } |x|\to\infty,
  \end{dcases}
  \label{eq:problem}
\end{align}
where $i$ is the imaginary unit, and the unit normal vector $\nu (x)$ on $\Gamma$ is directed into $\Omega_0$.
The setting of this problem is illustrated in Fig. \ref{fig:domain}.
\begin{figure}[tb]
  \centering
  \includegraphics[width=0.3\linewidth]{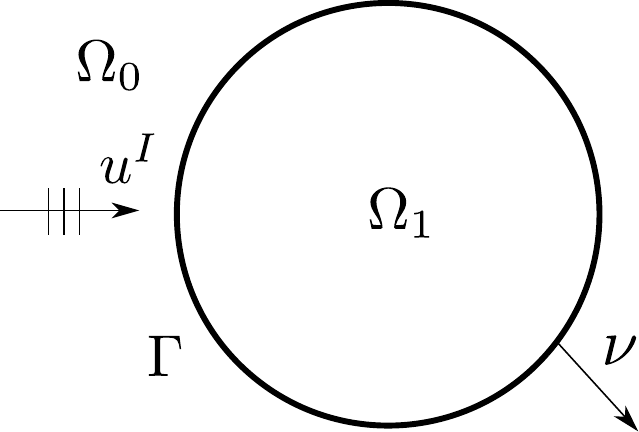}
  \caption{Diagram of the Helmholtz transmission problem}
  \label{fig:domain}
\end{figure}

\begin{remark}
  Under the assumption of this subsection,
  the transmission problem \eqref{eq:problem} is uniquely solvable \cite{kress1977transmission}.
  In fact, the definition of $u^{\mathrm{in}}$ is excessive to make Helmholtz transmission problems uniquely solvable.
  However, we limit ourselves to this case as we are interested in scattering problems.
\end{remark}

\subsection{Boundary integral operators} 
We define the following boundary integral operators:
\begin{align}
  (S_k\varphi)(x) &:= \int_{\Gamma} G(x,y;k)\varphi(y) \dd s(y),
  \\
  (D_k\varphi)(x) &:= \int_{\Gamma} \frac{\partial G}{\partial \nu(y)}(x,y;k)\varphi(y) \dd s(y),
  \\
  (D_k^* \varphi)(x) &:= \int_{\Gamma} \frac{\partial G}{\partial \nu(x)}(x,y;k)\varphi(y) \dd s(y),
  \\
  (N_k\varphi)(x) &:= \frac{\partial}{\partial\nu(x)}\int_{\Gamma} \frac{\partial G}{\partial \nu(y)}(x,y;k)\varphi(y) \dd s(y),
\end{align}
for all $x\in\Gamma$ and 
function $\varphi : \Gamma \to \mathbb{C}$, where
$\frac{\partial f}{\partial\nu(z)}$ is the normal derivative with respect to $\nu(z)$, 
defined as
\[
\frac{\partial f}{\partial\nu(z)} = \lim_{h\to 0_+} \nu(z)\cdot \nabla f(z + h \nu(z)), \quad z \in \Gamma.
\]
Here, $G(x, y; k)$ is the fundamental solution of the Helmholtz equation, given by
\begin{align}
  G(x, y; k) =
  \begin{dcases}
    \frac{i}{4} H_{0}^{(1)} (k |x-y|) & (d = 2), \\
    \frac{e^{i k |x - y|}}{4 \pi |x-y|} & (d = 3),
  \end{dcases}
\end{align}
for $x, y \in \mathbb{R}^d$ with $x \neq y$, where $H_{n}^{(1)}(z)$ is the $n$-th order Hankel function of the first kind for $n \in \mathbb{Z}$.

We first review the well-known mapping properties.
Although Lemma \ref{lem:N-N} is obvious,
it is proven here since the authors cannot find a direct description.
The embedding operator $\iota$ in Theorem \ref{prop:embedding}
has important roles in proving the well-posedness of a boundary integral operator in this paper.

\begin{thm}[Colton and Kress \cite{colton2013inverse}, Theorem 3.4, and Kress \cite{kress2014linear}, Theorems 7.5 and 7.6] \label{thm:mapping}
  Let $k>0$ and let $\Omega$ be an open and bounded subset of $\mathbb R^d$ with boundary $\Gamma$ of class $C^2$.
  Then,
  \begin{enumerate}
  \item $S_k:C^{0,\alpha}(\Gamma)\to C^{1,\alpha}(\Gamma)$ is bounded,
  \item $D_k:C^{1,\alpha}(\Gamma)\to C^{1,\alpha}(\Gamma)$ is compact,
  \item $D_k^*:C^{0,\alpha}(\Gamma)\to C^{0,\alpha}(\Gamma)$ is compact, and
  \item $N_k:C^{1,\alpha}(\Gamma)\to C^{0,\alpha}(\Gamma)$ is bounded.
  \end{enumerate}
\end{thm}

\begin{thm}[Colton and Kress \cite{colton1983integral}, Theorem 2.31, and Colton and Kress \cite{colton2013inverse}, \S 3.1] \label{prop:calderon}
  Let $k>0$ and let $\Omega$ be an open and bounded subset of $\mathbb R^d$ with boundary $\Gamma$ of class $C^2$.
  Then,
  \begin{enumerate}
  \item $N_k - N_0 : C^{0,\alpha}(\Gamma)\to C^{0,\alpha}(\Gamma)$ is compact,
  \item $S_k N_k = D_k D_k - \frac{1}{4} I^{1,\alpha}$,
  \item $N_kS_k = D^*_k D^*_k - \frac{1}{4} I^{0,\alpha}$,
  \item $D_k S_k = S_k D_k^*$, and
  \item $D_k^* N_k = N_k D_k$,
  \end{enumerate}
  where $I^{1,\alpha}$ and $I^{0,\alpha}$ are the identity operators
  on $C^{1, \alpha}(\Gamma)$ and $C^{0, \alpha}(\Gamma)$, respectively,
  and $N_0$ corresponds to the Laplace kernel version of $N_k$.
  Statements 2, 3, 4, and 5 are called Calder\'on relations.
\end{thm}

\begin{thm}[Gilbarg and Trudinger \cite{gilbarg1998elliptic}, Lemma 6.36] \label{prop:embedding}
  Let $\Omega$ be an open and bounded subset of $\mathbb R^d$ with boundary $\Gamma$ of class $C^{1, \alpha}$ for $0 < \alpha < 1$.
  Then, the embedding operator $\iota: C^{1,\alpha}(\Gamma)\to C^{0,\alpha}(\Gamma)$ is compact.
\end{thm}

\begin{lemma} \label{lem:N-N}
  Let $k_1>0$, $k_2>0$, and $0 < \alpha < 1$.
  Let $\Omega$ be an open and bounded subset of $\mathbb R^d$ with boundary $\Gamma$ of class $C^2$.
  Then, $N_{k_1} - N_{k_2}$ is compact from $C^{1,\alpha}(\Gamma) $ into $C^{0,\alpha}(\Gamma)$.
\end{lemma}
\begin{proof}
  From the relation
  $N_{k_1} - N_{k_2} = (N_{k_1} - N_{0}) - (N_{k_2} - N_{0})$,
  we can see that $N_{k_1} - N_{k_2}$ is compact
  from $C^{0,\alpha}(\Gamma)$ into $C^{0,\alpha}(\Gamma)$ by Theorem \ref{prop:calderon}.
  Considering $(N_{k_1} - N_{k_2}) (\iota \varphi$) for all $\varphi \in C^{1,\alpha}(\Gamma)$,
  we can also see that $N_{k_1} - N_{k_2}$ is a compact operator from $C^{1,\alpha}(\Gamma)$ into $C^{0,\alpha}(\Gamma)$
  since the embedding operator $\iota : C^{1,\alpha}(\Gamma)\to C^{0,\alpha}(\Gamma)$ is compact from Theorem \ref{prop:embedding}.
\end{proof}

\begin{thm}[Kress \cite{kress2014linear}, Theorem 3.4] \label{prop:riesz}
  Let $K : X \to X$ be a compact linear operator on a normed space $X$.
  Let $I$ be the identity operator on $X$.
  Then, $I-K$ is injective if and only if it is surjective.
\end{thm}

\section{Well-posedness of direct-indirect mixed Burton--Miller boundary integral equation} \label{sec:well-posed}

\subsection{Boundary integral equations}
We formulate the proposed boundary integral equation in the H\"older space setting.
We define $X$ as $C^{1, \alpha}(\Gamma) \times C^{0, \alpha}(\Gamma) \times C^{0, \alpha}(\Gamma)$.
Note that $X$ is a Banach space.
We wish to find a density function $\varphi$ in $C^{0, \alpha}(\Gamma)$ that satisfies
\begin{align}
  u_{1}^{\mathrm{sc}}(x) = \int_{\Gamma} G(x, y; k_1) \varphi(y) \dd s(y), \quad x \in \Omega_{1}.
  \label{eq:indirect_u}
\end{align}
We note that $u_{1}^{\mathrm{sc}} \in C^{1, \alpha}(\overline \Omega_1)$ \cite[Theorems 2.12 and 2.17]{colton1983integral}.
In \eqref{eq:indirect_u}, we consider the limit of a sequence $(x_n) \subset \Omega_1$ with $x_n \to y$ on $\Gamma$
to have the following relations:
\begin{align}
  u &= S_{k_1} \varphi \quad \text{on } \Gamma, \\
  q &= \frac{1}{\varepsilon_1} \qty(D_{k_1}^* + \frac{1}{2}) \varphi \quad \text{on } \Gamma,
\end{align}
where $u$ and $q$ have already been defined as the transmission condition in \eqref{eq:problem}.
Let $\beta$ be a constant in $\mathbb{C}$, called the Burton--Miller constant, such that $\Im(\beta) \neq 0$.
Then, Green's identity gives the following equation:
\begin{align}
  \qty(D_{k_0} -1/2 + \beta N_{k_0})u -\varepsilon_0 \left\{ S_{k_0} + \beta \left(D^*_{k_0}+\frac{1}{2}\right)\right\} q = - u^{\mathrm{in}} - \beta \pdv{u^{\mathrm{in}}}{\nu} \quad \text{on } \Gamma.
\end{align}
Considering these couplings, the direct-indirect mixed Burton--Miller boundary integral equation
is given by
\begin{equation}
  Ax=f,
  \label{eq:mixed_BM_equation}
\end{equation}
where $x$ and $f$ are functions on $X$ given by
\[
x := \mqty[
  u \\
  q \\
  \varphi
],
\]
\[
f := \mqty[
  0 \\
  0 \\
  -\iota u^{\mathrm{in}} - \beta \pdv{u^{\mathrm{in}}}{\nu}
],
\]
respectively, and
$A : X \to X$ is the bounded linear operator, defined as
\begin{multline}
  A:= \\
  \begin{bmatrix}
    I^{1, \alpha} & 0 & -S_{k_1}
    \\
    0 & I^{0, \alpha} & -\frac{1}{\varepsilon_1}\left(D^*_{k_1} + \frac{1}{2}I^{0, \alpha} \right)
    \\
    \iota D_{k_0} - \iota \frac{1}{2}I^{1, \alpha} + \beta N_{k_0} & -\varepsilon_0 \left\{ \iota S_{k_0} + \beta \left(D^*_{k_0}+\frac{1}{2}I^{0, \alpha}\right)\right\} & 0
  \end{bmatrix}
  .
  \label{eq:mixed_bm_op}
\end{multline}
Note that this boundary integral equation is obtained by rearranging a previously proposed equation \cite{YasuhiroMATSUMOTO202308-231124}.

\subsection{Well-posedness}
Following the standard argument, we show that the linear operator $A$ is injective and surjective to prove that it is invertible on $X=C^{1,\alpha}(\Gamma)\times C^{0,\alpha}(\Gamma)\times C^{0,\alpha}(\Gamma)$ (i.e., the boundary integral equation is well-posed).

\begin{lemma} \label{thm:injective}
  Let $A : X \to X$ be the integral operator defined by \eqref{eq:mixed_bm_op}.
  Let $\omega$, $\varepsilon_0$, and $\varepsilon_1$ be real positive constants.
  Then, $A$ is injective.
\end{lemma}
\begin{proof}
  Suppose that $\phi \in C^{1,\alpha}(\Gamma)$, $\psi \in C^{0,\alpha}(\Gamma)$, and $\eta \in C^{0,\alpha}(\Gamma)$ satisfy
  \begin{align}
    A
    \mqty[
      \phi \\
      \psi \\
      \eta
    ]
    =
    \mqty[
      0 \\
      0 \\
      0
    ].
    \label{eq:injec}
  \end{align}
  It suffices to show that the above equation has only the trivial solution.
  We can define two potentials $v$ and $w$ using $\phi$ and $\psi$, expressed as
  \begin{align}
    v(x) &:= \int_{\Gamma} \pdv{G}{\nu(y)}\qty(x, y; k_0)\phi(y) \dd{s(y)}
    - \varepsilon_{0} \int_{\Gamma} G(x, y; k_0) \psi(y) \dd{s(y)},
    \quad x \in \mathbb{R}^2 \setminus \Gamma, \\
    w(x) &:= -\int_{\Gamma} \pdv{G}{\nu(y)}\qty(x, y; k_1)\phi(y) \dd{s(y)}
    + \varepsilon_{1} \int_{\Gamma} G(x, y; k_1) \psi(y) \dd{s(y)},
    \quad x \in \mathbb{R}^2 \setminus \Gamma.
  \end{align}
  Let $g{|}_{\pm}$ be the limits of a function $g$, defined as
  \begin{align}
    g{|}_{+} := \lim_{h \to 0_{+}} g(x + h\nu(x)), \quad x \in \Gamma, \\
    g{|}_{-} := \lim_{h \to 0_{+}} g(x - h\nu(x)), \quad x \in \Gamma.
  \end{align}
  We have the following jump relations of the potential $v$ and its normal derivative:
  \begin{align}
    v|_{\pm} &= \qty(D_{k_0} \pm \tfrac{1}{2}) \phi(x) - \varepsilon_{0} S_{k_0} \psi(x), \quad x \in \Gamma, \label{eq:potv_trace} \\
    \eval{\pdv{v}{\nu}}_{\pm} &= N_{k_0} \phi(x) - \varepsilon_{0} \qty(D^{*}_{k_0} \mp \tfrac{1}{2}) \psi(x), \quad x \in \Gamma. \label{eq:potv_dev_trace}
  \end{align}
  From \eqref{eq:injec}, the linear combination satisfies $\iota \qty(v|_{-}) + \beta \eval{\pdv{v}{\nu}}_{-} = 0$.
  Then, $v = 0$ in $\Omega_1$ since the interior problem of the Helmholtz equation with the homogeneous impedance boundary condition
  \begin{align}
    \begin{dcases}
      \varDelta v + k_{0}^{2} v = 0, \quad x \in \Omega_1 \\
      v|_{-} + \beta \eval{\pdv{v}{\nu}}_{-} = 0, \quad x \in \Gamma,
    \end{dcases}
  \end{align}
  has only the trivial solution when $k_0 > 0$ and $\Im (\beta) \neq 0$.
  Therefore, we have
  \begin{gather}
    v = 0 \quad \text{ in } \Omega, \\
    v|_{-} = \eval{\pdv{v}{\nu}}_{-} = 0. \label{eq:zero_v_trace}
  \end{gather}
  Similar to the case of $v$, we have the following jump relations:
  \begin{align}
    w|_{\pm} &= -\qty(D_{k_1} \pm \tfrac{1}{2}) \phi(x) + \varepsilon_{1} S_{k_1} \psi(x), \quad x \in \Gamma, \label{eq:potw_trace} \\
    \eval{\pdv{w}{\nu}}_{\pm} &= -N_{k_1} \phi(x) + \varepsilon_{1} \qty(D^{*}_{k_1} \mp \tfrac{1}{2}) \psi(x), \quad x \in \Gamma. \label{eq:potw_dev_trace}
  \end{align}
  We extract from \eqref{eq:injec} the relations
  \begin{align}
    \phi &= S_{k_1} \eta, \quad x \in \Gamma, \label{eq:indirect_rep} \\
    \psi &= \tfrac{1}{\varepsilon_1}\qty(D^{*}_{k_1} + \tfrac{1}{2}) \eta, \quad x \in \Gamma. \label{eq:indirect_def_rep}
  \end{align}
  Substituting these relations into $w|_+$ and utilizing the Calder\'on relation $S_{k_1} D_{k_1}^{*} - D_{k_1} S_{k_1} = 0$ from Theorem \ref{prop:calderon},
  we can see that $w|_{+} = 0$, as follows:
  \begin{align}
    w|_{+} &= -\qty(D_{k_1} + \tfrac{1}{2}) \phi(x) + \varepsilon_{1} S_{k_1} \psi(x) \\
    &= -\qty(D_{k_1} + \tfrac{1}{2}) S_{k_1} \eta
    + \varepsilon_{1} S_{k_1} \tfrac{1}{\varepsilon_1}\qty(D^{*}_{k_1} + \tfrac{1}{2}) \eta \\
    &= \qty(-D_{k_1}S_{k_1} + S_{k_1} D_{k_1}^{*})\eta \\
    &= 0.
  \end{align}
  Then, $w = 0$ in $\Omega_0$ since the exterior problem of the Helmholtz equation with the homogeneous Dirichlet condition
  \begin{align}
    \begin{dcases}
      \varDelta w + k_{1}^{2} w = 0, \quad x \in \Omega_0, \\
      w|_{+} = 0, \quad x \in \Gamma, \\
      w(x) \text{ satisfies the radiation condition when } |x| \to \infty.
    \end{dcases}
  \end{align}
  has only the trivial solution when $k_1 > 0$.
  Therefore, the following conditions hold:
  \begin{gather}
    w = 0 \quad \text{in } \Omega_0, \\
    w|_{+} = \eval{\pdv{w}{\nu}}_{+} = 0. \label{eq:zero_w_trace}
  \end{gather}
  Furthermore, $v = 0$ in $\Omega_0$ and $w = 0$ in $\Omega_1$ are established via
  the unique solvability of the transmission problems.
  We can see from \eqref{eq:potv_trace} and \eqref{eq:potv_dev_trace} that
  \begin{align}
    v|_{+} -  v|_{-} &= \phi, \\
    \eval{\pdv{v}{\nu}}_{+} - \eval{\pdv{v}{\nu}}_{-} &= \varepsilon_{0} \psi.
  \end{align}
  Similarly, we can see from \eqref{eq:potw_trace} and \eqref{eq:potw_dev_trace} that
  \begin{align}
    w|_{+} -  w|_{-} &= \phi, \\
    \eval{\pdv{w}{\nu}}_{+} - \eval{\pdv{w}{\nu}}_{-} &= -\varepsilon_{1} \psi.
  \end{align}
  Applying \eqref{eq:zero_v_trace} and \eqref{eq:zero_w_trace} to the above relations,
  we obtain the transmission conditions.
  In summary, we have the homogeneous transmission problem
  \begin{align}
    \begin{dcases}
      \varDelta v + k_{0}^{2}v = 0, \quad x \in \Omega_0, \\
      \varDelta w + k_{1}^{2}w = 0, \quad x \in \Omega_1, \\
      v|_{+} = w|_{-}, \\
      \frac{1}{\varepsilon_0} \eval{\pdv{v}{\nu}}_{+} = \frac{1}{\varepsilon_1} \eval{\pdv{w}{\nu}}_{-}, \\
      v(x) \text{ satisfies the radiation condition when } |x| \to \infty,
    \end{dcases}
  \end{align}
  which has only the trivial solution $v = 0$ in $\Omega_0$ and $w = 0$ in $\Omega_1$.
  Therefore, $\phi = \psi = 0$.

  Finally we show that $\eta = 0$ from $\phi = \psi = 0$.
  From \eqref{eq:indirect_rep}, \eqref{eq:indirect_def_rep}, and $\phi = \psi = 0$, we have the following:
  \begin{align}
    0 &= S_{1} \eta, \quad x \in \Gamma, \label{eq:homo_indiret_rep} \\
    0 &= \tfrac{1}{\varepsilon_1}\qty( D^{*}_{1} + \tfrac{1}{2}) \eta, \quad x \in \Gamma. \label{eq:homo_indirect_def_rep}
  \end{align}
  For $\beta \in \mathbb{C}$ and $\Im(\beta) \neq 0$,
  we obtain a linear combination of \eqref{eq:homo_indiret_rep} and \eqref{eq:homo_indirect_def_rep}
  multiplied by $\beta$, expressed as
  \begin{align}
    \qty(S_{1} + \beta \tfrac{1}{\varepsilon_1} \qty(D^{*}_{1} + \tfrac{1}{2})) \eta = 0. \label{eq:homo_bm_dirichlet}
  \end{align}
  Due to the unique solvability of the Burton--Miller boundary integral equation with the homogeneous Dirichlet condition \cite{burton1971application, colton1983integral},
  \eqref{eq:homo_bm_dirichlet} has only the trivial solution.
  Therefore, $\eta = 0$.
  Thus, we have shown that only the trivial solution $\phi = \psi = \eta = 0$ is allowed.
\end{proof}

\begin{lemma}\label{lemma:fredholm}
  Under the same assumption as that in Lemma \ref{thm:injective},
  the boundary integral operator $A:X \to X$, defined by \eqref{eq:mixed_bm_op}, is surjective.
\end{lemma}
\begin{proof}
  We define the bounded linear operators $U$ and $V$ on $X$ as
  \[
  U := 
  \begin{bmatrix}
    I^{1,\alpha} & & \\
    & I^{0,\alpha} & \\
    -U_{1} & -U_{2} & I^{0,\alpha}
  \end{bmatrix}
  ,
  \quad
  V :=
  \begin{bmatrix}
    I^{1,\alpha} & & -V_{1}
    \\
    & I^{0,\alpha} & -V_{2}
    \\
    & & I^{0,\alpha}
  \end{bmatrix}
  ,
  \]
  where $U_{1}$, $U_{2}$, $V_{1}$, and $V_{2}$ are defined as
  \begin{align}
    U_{1} &:= \iota D_{k_0}  - \frac{1}{2}\iota + \beta N_{k_0}, \\
    U_{2} &:= -\varepsilon_0 \qty{ \iota S_{k_0}+\beta\left(D^*_{k_0}+\frac{1}{2}I^{0,\alpha}\right) }, \\
    V_{1} &:= -S_{k_1}, \\
    V_{2} &:= -\frac{1}{\varepsilon_1}\left( D^*_{k_1} + \frac{1}{2}I^{0,\alpha} \right).
  \end{align}
  It is clear that $U$ and $V$ are invertible on $X$. A straightforward calculation shows that
  \[
  UAV =
  \begin{bmatrix}
    I^{1,\alpha} & & \\
    & I^{0,\alpha} & \\
    -U_{1} & -U_{2} & I^{0,\alpha}
  \end{bmatrix}
  \begin{bmatrix}
    I^{1,\alpha} & & V_{1}\\
    & I^{0,\alpha} & V_{2}\\
    U_{1} & U_{2} & 0
  \end{bmatrix}
  \begin{bmatrix}
    I^{1,\alpha} & & -V_{1}
    \\
    & I^{0,\alpha} & -V_{2}
    \\
    & & I^{0,\alpha}
  \end{bmatrix}
  =
  \begin{bmatrix}
    I^{1,\alpha} & &
    \\
    & I^{0,\alpha} &
    \\
    & & F
  \end{bmatrix}
  ,
  \]
  where $F:C^{0,\alpha}(\Gamma)\to C^{0,\alpha}(\Gamma)$ is the linear operator, expressed as
  \begin{align*}
    F &= -U_{1} V_{1} -U_{2} V_{2} \\
    &= -\frac{\beta }{4} \qty(1 + \frac{\varepsilon_0}{\varepsilon_1}) I^{0,\alpha} + K.
  \end{align*}
  Here, 
  \begin{multline}
    K= \iota D_{k_0} S_{k_1} -\frac{1}{2} \iota S_{k_1} + \beta (N_{k_0} - N_{k_1})S_1 + \beta D_{k_1}^* D_{k_1}^* \\
    -\frac{\varepsilon_0}{\varepsilon_1} \left( \iota S_{k_0}D_{k_1}^* + \frac{1}{2} \iota S_{k_0} + \beta D_{k_0}^* D_{k_1}^* + \frac{1}{2}\beta D_{k_0}^* + \frac{1}{2} \beta D_{k_1}^{*} \right).
  \end{multline}
  Since $K$ is compact, we observe that $F$ is the sum of the identity operator multiplied by some constant and the compact operator.
  Since $A$ is injective from Lemma \ref{thm:injective},
  $UAV$ is also injective.
  Therefore, $UAV$ is surjective from Theorem \ref{prop:riesz}.
  From the bijectivity of $U$ and $V$ on $X$,
  $U^{-1} UAV V^{-1} = A$ is surjective.
\end{proof}

\begin{thm}
  Under the same assumption as that in Lemma \ref{thm:injective},
  the boundary integral equation \eqref{eq:mixed_BM_equation} is well-posed. 
\end{thm}
\begin{proof}
  From Lemmas \ref{thm:injective} and \ref{lemma:fredholm},
  $A$ is bijective.
  Moreover, $A^{-1}$ is bounded since $A$ is a bounded linear operator on the Banach space $X$.
  The well-posedness of the direct-indirect mixed Burton--Miller equation \eqref{eq:mixed_BM_equation} is proven.
\end{proof}

\section{Proposed direct solvers} \label{sec:solvers}
We now propose efficient direct solvers using the direct-indirect mixed Burton--Miller equation.
They are efficient compared to LU factorization for the ordinary Burton--Miller equation, which uses only direct unknowns $u$ and $q$.
We first introduce the efficient LU factorization and then apply it to the fast direct solver, which is a variant \cite{MATSUMOTO2025106148} of the Martinsson--Rokhlin type solver \cite{MARTINSSON20051}.

The ordinary Burton--Miller equation is
\begin{equation}
  \begin{bmatrix}
    \iota D_{k_0} - \iota \frac{1}{2}I^{1, \alpha} + \beta N_{k_0} & -\varepsilon_0 \left\{ \iota S_{k_0} + \beta \left(D^*_{k_0}+\frac{1}{2}I^{0, \alpha}\right)\right\} \\
    D_{k_1} + \frac{1}{2}I^{1, \alpha} & -S_{k_1}
  \end{bmatrix}
  \mqty[
    u \\
    q
  ] = 
  \mqty[
    -\iota u^{\mathrm{in}} - \beta \pdv{u^{\mathrm{in}}}{\nu}\\
    0
  ].
  \label{eq:ordi_BM_equation}
\end{equation}
Its integral operator is a bounded linear operator from $C^{1, \alpha}(\Gamma) \times C^{0, \alpha}(\Gamma)$ into $C^{0, \alpha}(\Gamma) \times C^{1, \alpha}(\Gamma)$.
For the discretization of the boundary integral equations \eqref{eq:mixed_BM_equation} and \eqref{eq:ordi_BM_equation},
we use the Nystr\"om method based on the zeta-corrected quadrature \cite{wu2023unified}.
The zeta-corrected quadrature is a local corrected equidistant trapezoidal quadrature that can handle singularities in boundary integral equations.
Since the variant of the Martinsson--Rokhlin type solver is based on the strategy of weak admissibility type low-rank approximation,
we set the number of local correction points to 1 for the zeta-corrected quadrature.
For a function $\phi$ on $\Gamma$, this zeta-corrected quadrature results in at least $O(h)$ convergence for $N_k \phi$ and at least $O(h^2)$ convergence for other layer potentials $S_k \phi$, $D_k \phi$, and $D_k^* \phi$ in two dimensions \cite{wu2023unified}.
Here, $h$ is the equispaced interval of integral points in a parameter space $[0, 2\pi]$.
In what follows, we only discuss the two-dimensional case and $I^{0, \alpha}$ and $I^{1, \alpha}$ are simply denoted as $I$.

The discretized boundary integral equations are obtained using the
Nystr\"om method based on the zeta-corrected quadrature as follows.
Let $N$ be the number of quadrature points
whose coordinates are equally spaced in parameterized $\Gamma$.
By the zeta-corrected quadrature \cite{wu2023unified}, $S_k \phi$, $D_k \phi$, $D_k^* \phi$, $N_k \phi$, and $I \phi$
are discretized
to multiplications of matrices
$\bm{S}_k$, $\bm{D}_k$,
$\bm{D^*}_k$, $\bm{N}_k \in \mathbb{C}^{N \times N}$,
and $\bm{I}$ of size $N \times N$
and a vector $\bm{\phi} \in \mathbb{C}^N$ for a function $\phi$ on $\Gamma$.
Let $\bm{u}$, $\bm{q}$, and $\bm{\varphi} \in \mathbb{C}^N$ be discretized versions of $u$, $q$, and $\varphi$, respectively.
Then, the discretized direct-indirect Burton--Miller equation is written as
\begin{multline}
  \begin{bmatrix}
    \bm{I} & 0 & -\bm{S}_{k_1}
    \\
    0 & \bm{I} & -\frac{1}{\varepsilon_1}\left(\bm{D^*}_{k_1} + \frac{1}{2}\bm{I} \right)
    \\
    \bm{D}_{k_0} - \frac{1}{2}\bm{I} + \beta \bm{N}_{k_0} & -\varepsilon_0 \left\{ \bm{S}_{k_0} + \beta \left( \bm{D^*}_{k_0}+\frac{1}{2} \bm{I}\right)\right\} & 0
  \end{bmatrix}
  \mqty[
    \bm{u} \\
    \bm{q} \\
    \bm{\varphi}
  ]
  \\
  =
  \mqty[
    0 \\
    0 \\
    -\bm{u^{\mathrm{in}}} - \beta \bm{q^{\mathrm{in}}}
  ],
  \label{eq:dis_mixedBM}
\end{multline}
and the discretized ordinary Burton--Miller equation is expressed as
\begin{equation}
  \begin{bmatrix}
    \bm{D}_{k_0} - \frac{1}{2}\bm{I} + \beta \bm{N}_{k_0} & -\varepsilon_0 \left\{ \bm{S}_{k_0} + \beta \left(\bm{D^*}_{k_0}+\frac{1}{2}\bm{I} \right)\right\} \\
    \bm{D}_{k_1} + \frac{1}{2}\bm{I} & -\bm{S}_{k_1}
  \end{bmatrix}
  \mqty[
    \bm{u} \\
    \bm{q}
  ] = 
  \mqty[
    -\bm{u^{\mathrm{in}}} - \beta \bm{q^{\mathrm{in}}} \\
    0
  ],
  \label{eq:dis_BM}
\end{equation}
where $\bm{u}^{\mathrm{in}}$ and $\bm{q}^{\mathrm{in}}$ are vectors in $\mathbb{C}^N$ corresponding to the discretized $u^{\mathrm{in}}$ and $\pdv{u^{\mathrm{in}}}{\nu}$, respectively.

\subsection{Efficient LU factorization} \label{sec:LU}
The discretized direct-indirect Burton--Miller equation \eqref{eq:dis_mixedBM} has $3N$ degrees of freedom,
which is more than the $2N$ degrees of freedom of the discretized ordinary Burton--Miller equation \eqref{eq:dis_BM}.
However, we can make the direct solvers for \eqref{eq:dis_mixedBM} be faster than those for \eqref{eq:dis_BM} since for the former equation, the diagonal blocks are an identity or zero matrix and the placement of the integral operators is sparse.
First, we propose an efficient direct solver for \eqref{eq:dis_mixedBM} based on LU-type factorization.
Note that the standard LU factorization of \eqref{eq:dis_BM} requires approximately $\frac{2}{3}(2N)^3 = \frac{16}{3} N^3$ floating-point operations.
\begin{thm} \label{thm:LU}
  Let the discretized direct-indirect Burton--Miller equation \eqref{eq:dis_mixedBM} be denoted as
  \begin{equation}
    \begin{bmatrix}
      \bm{I} & 0 & M_{13}
      \\
      0 & \bm{I} & M_{23}
      \\
      M_{31} & M_{32} & 0
    \end{bmatrix}
    \begin{bmatrix}
      \bm{u} \\
      \bm{q} \\
      \bm{\varphi}
    \end{bmatrix}
    =
    \begin{bmatrix}
      0 \\
      0 \\
      b_3
    \end{bmatrix}
    ,
    \label{eq:given_system}
  \end{equation}
  where each matrix or vector is given by
  \begin{align}
    &M_{13} = -\bm{S}_{k_1}, \quad
    M_{23} = -\frac{1}{\varepsilon_1}\left(\bm{D^*}_{k_1} + \frac{1}{2}\bm{I} \right), \\
    &M_{31} = \bm{D}_{k_0} - \frac{1}{2}\bm{I} + \beta \bm{N}_{k_0}, \quad
    M_{32} = -\varepsilon_0 \left\{ \bm{S}_{k_0} + \beta \left( \bm{D^*}_{k_0}+\frac{1}{2} \bm{I}\right)\right\}, \\
    &b_3 = -\bm{u^{\mathrm{in}}} - \beta \bm{q^{\mathrm{in}}}.
  \end{align}
  Assume that the coefficient matrix of \eqref{eq:given_system} is invertible.
  Then, the solution to these linear equations is given by
  \[
  \mqty[
    \bm{u} \\
    \bm{q} \\
    \bm{\varphi}
  ]
  =
  \mqty[
    -M_{13} x_3 \\
    -M_{23} x_3 \\
    x_3
  ],
  \]
  where $x_3 = (-M_{31} M_{13} - M_{32} M_{23})^{-1} b_3$.
  This solution procedure requires approximately $\frac{14}{3} N^3$ floating-point operations.
\end{thm}
\begin{proof}
  Since the coefficient matrix of the given system of linear equations is invertible from the assumption,
  the following LU factorization exists:
  \[
  \begin{bmatrix}
    \bm{I} & 0 & M_{13}
    \\
    0 & \bm{I} & {M_{23}}
    \\
      {M_{31}} & {M_{32}} & 0
  \end{bmatrix}
  =
  \begin{bmatrix}
    \bm{I} &  & 
    \\
      {L_{21}} & \bm{I} & 
      \\
        {L_{31}} & {L_{32}} & \bm{I}
  \end{bmatrix}
  \begin{bmatrix}
    {U_{11}} & {U_{12}} & {U_{13}}
    \\
    & {U_{22}} & {U_{23}}
    \\
    &  & {U_{33}}
  \end{bmatrix}
  \]
  This can be interpreted as an identity for the nine block matrices on the left-hand side.
  Immediately it follows that
  \begin{align}
    &U_{11} = \bm{I}, \quad U_{12} = 0, \quad U_{13} = M_{13}, \\
    &L_{21} = 0, \quad U_{22} = \bm{I}, \quad U_{23} = M_{23}, \\
    &L_{31} = M_{31}, \quad L_{32} = M_{32}, \quad U_{33} = -M_{31}M_{13} - M_{32}M_{23}.
  \end{align}
  From the above relations, we have the factorized form of \eqref{eq:given_system} expressed as
  \[
  \begin{bmatrix}
    \bm{I} &  & \\
    0 & \bm{I} & \\
    {M_{31}} & {M_{32}} & \bm{I}
  \end{bmatrix}
  \begin{bmatrix}
    \bm{I} & 0 & {M_{13}} \\
    & \bm{I} & {M_{23}} \\
    &  & {-M_{31}M_{13} - M_{32}M_{23}}
  \end{bmatrix}
  \begin{bmatrix}
    \bm{u} \\
    \bm{q} \\
    \bm{\varphi}
  \end{bmatrix}
  =
  \begin{bmatrix}
    0 \\
    0 \\
    b_3
  \end{bmatrix}
  .
  \]
  Let $L$, $U$, $x$, and $b$ be
  \[
  L =   \begin{bmatrix}
    \bm{I} &  & \\
    0 & \bm{I} & \\
    {M_{31}} & {M_{32}} & \bm{I}
  \end{bmatrix}
  , \quad
  U =
  \begin{bmatrix}
    \bm{I} & 0 & {M_{13}} \\
    & \bm{I} & {M_{23}} \\
    &  & {-M_{31}M_{13} - M_{32}M_{23}}
  \end{bmatrix}
  , \quad
  x = 
  \begin{bmatrix}
    \bm{u} \\
    \bm{q} \\
    \bm{\varphi}
  \end{bmatrix}
  , \quad
  b = 
  \begin{bmatrix}
    0 \\
    0 \\
    b_3
  \end{bmatrix}
  .
  \]
  Then, we can solve the factorized form of \eqref{eq:given_system}
  \[
  LUx=b
  \]
  using forward and backward substitution.
  Suppose that $y = Ux$ and
  \[
  y =
  \begin{bmatrix}
    y_1 \\
    y_2 \\
    y_3
  \end{bmatrix}
  ,
  \]
  where $y_1$, $y_2$, and $y_3 \in \mathbb{C}^N$.
  Considering the forward substitution for $Ly = b$, that is,
  \[
  \begin{bmatrix}
    \bm{I} &  & \\
    0 & \bm{I} & \\
    {M_{31}} & {M_{32}} & \bm{I}
  \end{bmatrix}
  \begin{bmatrix}
    y_1 \\
    y_2 \\
    y_3
  \end{bmatrix}
  =
  \begin{bmatrix}
    0 \\
    0 \\
    b_3
  \end{bmatrix}
  ,
  \]
  we have
  \begin{align}
    &y_1 = 0, \\
    &y_2 = 0, \\
    &y_3 = b_3
    .
  \end{align}
  Next, we consider the backward substitution for $Ux = y$, that is,
  \begin{equation}
    \begin{bmatrix}
      \bm{I} & 0 & {M_{13}} \\
      & \bm{I} & {M_{23}} \\
      &  & {-M_{31}M_{13} - M_{32}M_{23}}
    \end{bmatrix}
    \begin{bmatrix}
      \bm{u} \\
      \bm{q} \\
      \bm{\varphi}
    \end{bmatrix}
    =
    \begin{bmatrix}
      0 \\
      0 \\
      b_3
    \end{bmatrix}
    .
    \label{eq:bottom}
  \end{equation}
  We observe that $(-M_{31}M_{13} - M_{32}M_{23})$ is also invertible 
  since it can be obtained in the Gaussian elimination of the invertible coefficient of \eqref{eq:given_system}.
  Therefore, we have
  \[
  \bm{\varphi} = {(-M_{31}M_{13} - M_{32}M_{23})}^{-1} b_3
\]
  from the bottom relation in \eqref{eq:bottom}.
  Let $x_3 = {(-M_{31}M_{13} - M_{32}M_{23})}^{-1} b_3$.
  We obtain the solution of \eqref{eq:given_system} as
  \[
  \mqty[
    \bm{u} \\
    \bm{q} \\
    \bm{\varphi}
  ]
  =
  \mqty[
    -M_{13} x_3 \\
    -M_{23} x_3 \\
    x_3
  ].
  \]

  Finally, we estimate the leading order of the number of floating-point operations required.
  It is sufficient to consider the constant coefficient of $N^3$-order operations.
  Such operations appear only in the calculation $x_3 = {(-M_{31}M_{13} - M_{32}M_{23})}^{-1} b_3$.
  This calculation involves two multiplications of $N \times N$ matrices and one LU factorization of an $N \times N$ matrix.
  Therefore, we can see that
  the constant coefficient of the leading order of the number of floating-point operations is
  $2 \times (2N^3) + \frac{2}{3} N^3 = \frac{14}{3}N^3$.
\end{proof}

The approximate number of required floating-point operations of LU factorization for 
the discretized direct-indirect Burton--Miller equation \eqref{eq:dis_mixedBM} ($\frac{14}{3}N^3$)
is $12.5 \%$ smaller than that for the discretized ordinary Burton--Miller equation \eqref{eq:dis_BM} ($\frac{16}{3}N^3$).
Thus, the method in Theorem \ref{thm:LU} is expected to solve \eqref{eq:dis_mixedBM} more efficiently
than standard LU factorization of \eqref{eq:dis_BM}.

\subsubsection{Extension to more general right-hand sides}
We can extend Theorem \ref{thm:LU} to more cases.
First, we extend Theorem \ref{thm:LU} to the case where the right-hand side is an arbitrary vector in $\mathbb{C}^{3N}$.
This extension is used in the application to the Sakurai--Sugiura projection method to solve nonlinear eigenvalue problems in Section \ref{sec:ssm}.
The problem setting is given as follows:
\begin{equation}
  \begin{bmatrix}
    \bm{I} & 0 & M_{13}
    \\
    0 & \bm{I} & M_{23}
    \\
    M_{31} & M_{32} & 0
  \end{bmatrix}
  \begin{bmatrix}
    x_{1} \\
    x_{2} \\
    x_{3}
  \end{bmatrix}
  =
  \begin{bmatrix}
    b_{1} \\
    b_{2} \\
    b_{3}
  \end{bmatrix}
  ,
  \label{eq:for_ssm_RHS}
\end{equation}
where $M_{13}$, $M_{23}$, $M_{31}$, and $M_{32}$ were given in Theorem \ref{thm:LU},
$b_{i}$ is an arbitrary vector in $\mathbb{C}^N$ for $i = 1, 2, 3$,
and $x_{i}$ is the solution of the above linear equation.
\begin{cor} \label{lem:for_ssm_RHS}
  Assume that the coefficient matrix of \eqref{eq:for_ssm_RHS} is invertible.
  Then, the solution $x_{i}$ in \eqref{eq:for_ssm_RHS} is given by
  \[
  \begin{bmatrix}
    x_{1} \\
    x_{2} \\
    x_{3}
  \end{bmatrix}
  =
  \mqty[
    b_1 - M_{13}z_3 \\
    b_2 - M_{23}z_3 \\
    z_3
  ],
  \]
  where $z_3$ is expressed as 
  \[
  z_3 = (-M_{31} M_{13} - M_{32} M_{23})^{-1} (b_3 - M_{31}b_1 - M_{32}b_2).
  \]
\end{cor}
We next extend Theorem \ref{thm:LU}
to the case where the right-hand side is a block diagonal matrix toward the application to the fast direct solver.
Using the same notation for $M_{13}$, $M_{23}$, $M_{31}$, and $M_{32}$ as that in Theorem \ref{thm:LU},
the extended equation is given by
\begin{equation}
  \begin{bmatrix}
    \bm{I} & 0 & M_{13}
    \\
    0 & \bm{I} & M_{23}
    \\
    M_{31} & M_{32} & 0
  \end{bmatrix}
  \begin{bmatrix}
    X_{11} & X_{12} & X_{13} \\
    X_{21} & X_{22} & X_{23} \\
    X_{31} & X_{32} & X_{33}
  \end{bmatrix}
  =
  \begin{bmatrix}
    B_1 & &\\
    & B_2 &\\
    & & B_3
  \end{bmatrix}
  ,
  \label{eq:given_system_blockRHS}
\end{equation}
where $B_{1}$, $B_{2}$, and $B_{3}$ are matrices in $\mathbb{C}^{N \times c_1}$, $\mathbb{C}^{N \times c_2}$, and $\mathbb{C}^{N \times c_3}$ for some integers $c_1, c_2, c_3 > 0$, respectively.
Here,
$X_{11}$, $X_{21}$, and $X_{31}$ are matrices in $\mathbb{C}^{N \times c_1}$,
$X_{12}$, $X_{22}$, and $X_{32}$ are matrices in $\mathbb{C}^{N \times c_2}$,
and
$X_{13}$, $X_{23}$, and $X_{33}$ are matrices in $\mathbb{C}^{N \times c_3}$.

\begin{cor} \label{lem:block_LU}
Assume that the coefficient matrix of \eqref{eq:given_system_blockRHS} is invertible.
  Then, the solution $X_{ij}$ in \eqref{eq:given_system_blockRHS} is given by
  \[
  \begin{bmatrix}
    X_{11} & X_{12} & X_{13} \\
    X_{21} & X_{22} & X_{23} \\
    X_{31} & X_{32} & X_{33}
  \end{bmatrix}
  =
  \mqty[
    B_1 - M_{13}Z_{31} & -M_{13}Z_{32} & -M_{13}Z_{33} \\
    - M_{23}Z_{31} & B_2 -M_{23}Z_{32} & -M_{23}Z_{33} \\
    Z_{31} & Z_{32} & Z_{33} \\
  ],
  \]
  where $Z_{31}$, $Z_{32}$, and $Z_{33}$ are expressed as 
  \begin{align}
    &Z_{31} = -(-M_{31} M_{13} - M_{32} M_{23})^{-1} M_{31} {B_1}, \\
    &Z_{32} = -(-M_{31} M_{13} - M_{32} M_{23})^{-1} M_{32} {B_2}, \\
    &Z_{33} = (-M_{31} M_{13} - M_{32} M_{23})^{-1} {B_3}. \\
  \end{align}
\end{cor}
\begin{proof}
  The claims of Corollaries \ref{lem:for_ssm_RHS} and \ref{lem:block_LU} can be obtained by straightforward calculation in the same way as done for Theorem \ref{thm:LU}.
\end{proof}

\subsection{Application to fast direct solver} \label{sec:fds}
We now construct a variant of the Martinsson--Rokhlin fast direct solver \cite{MATSUMOTO2025106148} that uses the discretized direct-indirect mixed Burton--Miller equation \eqref{eq:dis_mixedBM}.
The speedup factors for this formulation exist in both the low-rank approximation method for the off-diagonal blocks of the linear equations and the compression technique for the linear equations; therefore, we discuss them sequentially.
The result of Section \ref{sec:LU} is used in the latter discussion.

\subsubsection{Low-rank approximation using proxy method}
For constructing a fast method to solve \eqref{eq:dis_mixedBM},
a low-rank approximation technique is effective.
We first describe the blocking of the system of linear equations \eqref{eq:dis_mixedBM} and
then discuss the low-rank approximation with the weak admissibility condition
for the off-diagonal blocks of this system using the proxy method.

The blocking method is essentially the same as that in \cite{MATSUMOTO2025106148},
differing only by the choice of Nystr\"om discretization versus Galerkin discretization;
nevertheless, it is briefly described here for ease of reference.
We apply binary tree decomposition to the row and column indices of the coefficient matrix.
Since the discretized direct-indirect Burton--Miller equation has $3 \times 3$ block structured coefficient,
we need three index sets for row blocks and three index sets for column blocks.
These binary trees include the indices corresponding to the coordinates
of the collocation points as the row-side indices, and the coordinates of the integration points as the column-side indices.
The term ``collocation points'' is not typically used in the context of the Nystr\"om method.
However, we use this to explain the proxy method.
Suppose that this tree structure is a perfect binary tree whose root is level 0.
Let the cells of the six trees be
$\{ J_{i, \mathrm{row}}^{1, L} \}_{i = 1}^{p}$,
$\{ J_{i, \mathrm{row}}^{2, L} \}_{i = 1}^{p}$,
$\{ J_{i, \mathrm{row}}^{3, L} \}_{i = 1}^{p}$,
$\{ J_{i, \mathrm{col}}^{1, L} \}_{i = 1}^{p}$,
$\{ J_{i, \mathrm{col}}^{2, L} \}_{i = 1}^{p}$, and 
$\{ J_{i, \mathrm{col}}^{3, L} \}_{i = 1}^{p}$,
respectively, where $p = 2^{L}$ is the number of cells at leaf level $L$ of the trees.
Assume that each cell has $n$ indices at leaf level $L$ for simplicity.
Here, these cells are sequences of indices
such that the $s$-th indices in the cells are expressed as
\begin{align}
  J_{i, \mathrm{row}}^{j, L}(s) &= s + \qty(\sum_{k = 1}^{i - 1} |J_{k, \mathrm{row}}^{j, L}|), \quad j = 1, 2, 3, \\
  J_{i, \mathrm{col}}^{j, L}(s) &= s + \qty(\sum_{k = 1}^{i - 1} |J_{k, \mathrm{col}}^{j, L}|), \quad j = 1, 2, 3,
\end{align}
for $i = 1, 2, \ldots, p$ and for $s = 1, 2, \ldots, n$, where $| Z |$ is the number of elements in the sequence $ Z $,
and the summation for $i = 1$ in the above relations is the empty sum.
The index sets $J_{i, \mathrm{row}}^{j, L}$ and $J_{i, \mathrm{col}}^{j, L}$ for $j = 1, 2, 3$ are identical at the leaf level.
In this tree structure, the index sets at the upper levels,
namely $\{ J_{i, \mathrm{row}}^{j, \ell} \}_{i = 1}^{2^{\ell}}$ and $\{ J_{i, \mathrm{col}}^{j, \ell} \}_{i = 1}^{2^{\ell}}$ for $\ell = L-1, L-2, \ldots, 0$ and $j = 1, 2, 3$,
are defined as subsequences of the immediately lower level, which consists of the skeletons selected by the proxy method.
We define the sets of the cell index sets as
\begin{align}
&J_{i, \mathrm{row}}^{\ell} := \{J_{i, \mathrm{row}}^{1, L}, J_{i, \mathrm{row}}^{2, L}, J_{i, \mathrm{row}}^{3, L} \}, \quad \ell = L, L-1, \ldots, 0, \\
&J_{i, \mathrm{col}}^{\ell} := \{J_{i, \mathrm{col}}^{1, L}, J_{i, \mathrm{col}}^{2, L}, J_{i, \mathrm{col}}^{3, L} \}, \quad \ell = L, L-1, \ldots, 0,
\end{align}
for each level $\ell$.

The system of the discretized direct-indirect mixed Burton--Miller equation of size $3N \times 3N$
is partitioned into blocks by the cells of the tree at leaf level $L$ as follows:
\begin{equation}
  \mqty[
  A_{11}^{L} & A_{12}^{L} & \cdots & A_{1p}^{L} \\
  A_{21}^{L} & A_{22}^{L} & \cdots & A_{2p}^{L} \\
  \vdots & \vdots &\ddots &\vdots \\
  A_{p1}^{L} & A_{p2}^{L} & \cdots & A_{pp}^{L}
  ]
  \mqty[
  x_1^{L} \\
  x_2^{L} \\
  \vdots \\
  x_p^{L}
  ]
  =
  \mqty[
  f_1^{L} \\
  f_2^{L} \\
  \vdots \\
  f_p^{L}
  ],
  \label{eq:block_linear}
\end{equation}
where each $A_{ij}^{L}$ corresponds to the interactions between cells
$J_{i, \mathrm{row}}^{L}$ and $J_{j, \mathrm{col}}^{L}$ for $i, j = 1, 2, \ldots, p$.
Similarly, $x_{i}$ corresponds to the solution on the integration points in $J_{i, \mathrm{col}}^{L}$
and $f_{i}$ corresponds to the discretized incident wave on the collocation points in $J_{i, \mathrm{row}}^{L}$.
More precisely, $A_{ij}^{L}$ is defined as
\begin{equation}
  A_{ij}^{L} :=
  \begin{bmatrix}
    \delta_{ij}\bm{I} & 0 & -\qty[\bm{S}_{k_1}]_{ij}^{L}
    \\
    0 & \delta_{ij}\bm{I} & -\frac{1}{\varepsilon_1}\left( \qty[\bm{D^*}_{k_1}]_{ij}^{L} + \delta_{ij}\frac{1}{2}\bm{I} \right)
    \\
    \qty[ \bm{D}_{k_0} ]_{ij}^{L} - \frac{1}{2}\delta_{ij}\bm{I} + \beta \qty[ \bm{N}_{k_0} ]_{ij}^{L} & -\varepsilon_0 \left\{ \qty[\bm{S}_{k_0}]_{ij}^{L} + \beta \left( \qty[\bm{D^*}_{k_0}]_{ij}^{L} +\frac{1}{2} \delta_{ij}\bm{I}\right)\right\} & 0
  \end{bmatrix}
  .
  \label{eq:block_Aij}
\end{equation}
Here, $A_{ij}^{L}$ is defined as
the interactions between cell indices $J_{i, \mathrm{row}}^L$ and $J_{j, \mathrm{col}}^L$.
For three row blocks of $A_{ij}^{L}$,
$J_{i, \mathrm{row}}^{1, L}$, 
$J_{i, \mathrm{row}}^{2, L}$, and 
$J_{i, \mathrm{row}}^{3, L}$ are used for the first, second, and third row blocks, respectively.
Similarly, for three column blocks of $A_{ij}^{L}$,
$J_{i, \mathrm{col}}^{1, L}$, 
$J_{i, \mathrm{col}}^{2, L}$, and 
$J_{i, \mathrm{col}}^{3, L}$ are used for the first, second, and third column blocks, respectively.
The notation $\qty[Z]_{ij}^{L}$ in \eqref{eq:block_Aij} stands for a block matrix of size $n \times n$
that is a subset of the whole $3N \times 3N$ coefficient matrix at leaf level $L$,
where $Z$ is $\bm{S}_{k}$, $\bm{D}_{k}$, $\bm{D^*}_{k}$, or $\bm{N}_{k}$.
In \eqref{eq:block_Aij}, $\bm{I}$ is an identity matrix of size $n$
and $\delta_{ij}$ is the Kronecker delta.
The size of identity matrix $\bm{I}$ is not specified as it can be determined from the context.
In \eqref{eq:block_linear}, $x_i \in \mathbb{C}^{3n}$ is defined as
\begin{equation}
  x_i^L :=
    \mqty[
      \bm{u}_{i}^L \\
      \bm{q}_{i}^L \\
      \bm{\varphi}_{i}^L
    ],
\end{equation}
where the vectors 
$\bm{u}_i^{L} \in \mathbb{C}^{n}$,
$\bm{q}_i^{L} \in \mathbb{C}^{n}$, and
$\bm{\varphi}_i^{L} \in \mathbb{C}^{n}$
are subsets of $\bm{u}$, $\bm{q}$, and $\bm{\varphi}$ corresponding to $J_{i, \mathrm{col}}^{L}$, respectively.
Similarly, $f_i \in \mathbb{C}^{3n}$ is defined as
\begin{equation}
  f_i^L :=
    \mqty[
      0 \\
      0 \\
      -(\bm{u^{\mathrm{in}}})_i^{L} - \beta (\bm{q^{\mathrm{in}}})_i^{L}
    ],
    \label{eq:rhs_f}
\end{equation}
where $0$ represents a zero vector of size $n$ and
$(\bm{u^{\mathrm{in}}})_i^{L}$ and $(\bm{q^{\mathrm{in}}})_i^{L}$ are subsets of 
$\bm{u^{\mathrm{in}}}$ and $\bm{q^{\mathrm{in}}}$ corresponding to $J_{i, \mathrm{row}}^L$, respectively.

We need low-rank approximation of all off-diagonal blocks of \eqref{eq:block_linear} such that
each $A_{ij}^{L}$ $(i \neq j)$ is approximated as
\begin{multline}
  A_{ij}^{L} \approx L_{i}^L S_{ij}^L R_{j}^L
  = \mqty[
    L_i^{1, L} & & \\
    & L_i^{2, L} & \\
    & & L_i^{3, L}
  ] \\
\mqty[
   &  & -\qty[\bm{S}_{k_1}]_{\tilde{i}\tilde{j}}^{L} \\
   &  & -\frac{1}{\varepsilon_1} \qty[\bm{D^*}_{k_1}]_{\tilde{i}\tilde{j}}^{L} \\
  \qty[ \bm{D}_{k_0} ]_{\tilde{i}\tilde{j}}^{L} + \beta \qty[ \bm{N}_{k_0} ]_{\tilde{i}\tilde{j}}^{L} & -\varepsilon_0 \left\{ \qty[\bm{S}_{k_0}]_{\tilde{i}\tilde{j}}^{L} + \beta \qty[\bm{D^*}_{k_0}]_{\tilde{i}\tilde{j}}^{L} \right\} & 
  ] \\
\mqty[
    R_i^{1, L} & & \\
    & R_i^{2, L} & \\
    & & R_i^{3, L}
  ], \quad i \neq j,
\label{eq:aij=usv}
\end{multline}
where $L_i^{1, L}$,  $L_i^{2, L}$, and $L_i^{3, L}$ $\in \mathbb{C}^{n \times k}$ and
$R_i^{1, L}$,  $R_i^{2, L}$, and $R_i^{3, L}$ $\in \mathbb{C}^{k \times n}$
are coefficients of a low-rank approximation calculated using the proxy method (described below).
Here, $k$ is the number of skeletons $k \ll n$.
$S_{ij}^L$ (of $L_{i}^L S_{ij}^L R_{j}^L$) in \eqref{eq:aij=usv} is the skeleton interactions at leaf level $L$ between
the skeleton indices $\tilde{J}_{i, \mathrm{row}}^{L}$ of $J_{i, \mathrm{row}}^{L}$ and $\tilde{J}_{j, \mathrm{col}}^{L}$ of $J_{j, \mathrm{col}}^{L}$.
The subscripts $\tilde{i}$ and  $\tilde{j}$ correspond to $\tilde{J}_{i, \mathrm{row}}^{L}$ and $\tilde{J}_{j, \mathrm{col}}^{L}$, respectively.
Here, note that $\tilde{J}_{i, \mathrm{row}}^{L}$ is the set $\{ \tilde{J}_{i, \mathrm{row}}^{1, L}, \tilde{J}_{i, \mathrm{row}}^{2, L}, \tilde{J}_{i, \mathrm{row}}^{3, L} \}$,
$\tilde{J}_{i, \mathrm{col}}^{L}$ is the set $\{ \tilde{J}_{i, \mathrm{col}}^{1, L}, \tilde{J}_{i, \mathrm{col}}^{2, L}, \tilde{J}_{i, \mathrm{col}}^{3, L} \}$,
and $L_{i}^{L}$ and $R_{i}^{L}$ are shared by the $i$-th row and column blocks, respectively, for $i = 1, 2, \ldots, p$.
The sequences $\tilde{J}_{i, \mathrm{row}}^{j, L}$ and $\tilde{J}_{i, \mathrm{col}}^{j, L}$, for $j = 1, 2, 3$, are subsequences of 
$J_{i, \mathrm{row}}^{j, L}$ and $J_{i, \mathrm{col}}^{j, L}$, respectively.

\begin{figure}[tb]
  \centering
  \includegraphics[width=0.4\linewidth]{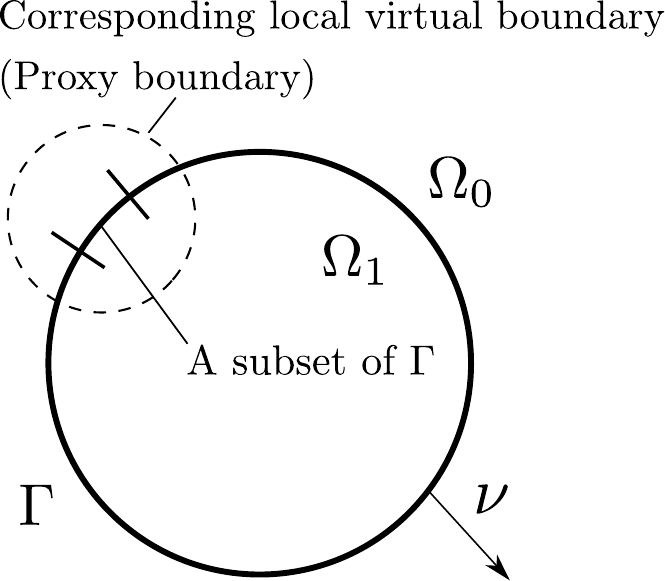}
  \caption{
    Local virtual boundary (proxy boundary) in the proxy method
  }
  \label{fig:proxy}
\end{figure}
We now describe the proxy method for computing the skeleton indices $\tilde{J}_{i, \mathrm{row}}^{L}$ and $\tilde{J}_{i, \mathrm{col}}^{L}$,
and the corresponding coefficients $L_{i}^L$ and $R_{i}^L$ at the leaf level.
The proxy method in this study is similar to the original method \cite{MARTINSSON20051}.
Here, we explain the difference from the original method by clarifying what are calculated as the interactions
between a local virtual boundary and a subset of $\Gamma$ in the proposed boundary integral equation.
Under the Nystr\"om discretization in this study,
a subset of $\Gamma$ corresponds to the coordinates of the collocation and integration points in cell $J_{i, \mathrm{row}}^{L}$ or $J_{i, \mathrm{col}}^{L}$, respectively.
We consider a local virtual boundary as a circle enclosing this subset of $\Gamma$, as shown in Fig. \ref{fig:proxy}.
The diameter of this local virtual boundary is set to about 1.5 times the diameter of the subset of $\Gamma$ corresponding to $J_{i, \mathrm{row}}^{L}$ or $J_{i, \mathrm{col}}^{L}$.
Set some equispaced points on the local virtual boundary.
Let $J_{i, \mathrm{row}}^{\prime L}$ be an index set corresponding to $J_{i, \mathrm{col}}^{L}$, which consists of the sum of some equispaced points on the local virtual boundary and the collocation points in the adjacent cell of $J_{i, \mathrm{col}}^{L}$ that are enclosed by the local virtual boundary.
Similarly, let $J_{i, \mathrm{col}}^{\prime L}$ be an index set corresponding to $J_{i, \mathrm{row}}^{L}$, which consists of the sum of some equispaced points on the local virtual boundary and the integration points in the adjacent cell of $J_{i, \mathrm{row}}^{L}$ that are enclosed by the local virtual boundary.
Assume that the number of elements in $J_{i, \mathrm{col}}^{\prime L}$ is equal to that in $J_{i, \mathrm{row}}^{\prime L}$ for $i = 1, 2, \ldots, p$.
Let $n^{\prime}$ be the number of elements in $J_{i, \mathrm{col}}^{\prime L}$ and $J_{i, \mathrm{row}}^{L}$.
Since, for a fixed $i$, the computational load for calculating
the interactions between $J_{i, \mathrm{row}}^{L}$ and $J_{j, \mathrm{col}}^{L}$ for all $j \in \{1, 2, \ldots, p \}$ is large,
we consider replacing this computation with the interactions between $J_{i, \mathrm{row}}^{L}$ and $J_{i, \mathrm{col}}^{\prime, L}$
by the proxy method.
This replacement reduces the required computation from global interactions to local interactions.
These local interactions are used only for the construction of the low-rank approximations of off-diagonal blocks.
The local interactions that should be evaluated are summarized in Table \ref{tb:low-rank}.
In Table \ref{tb:low-rank}, the subscript $i^{\prime} i$, 
such as in $\qty[ \bm{D}_{k_0} ]_{i^{\prime} i}^{L}$, 
indicates the interaction between $J_{i, \mathrm{row}}^{\prime L}$ and $J_{i, \mathrm{col}}^{L}$ for a fixed $i \in \{ 1, 2, \ldots, p \}$.
Similarly, the subscript $i i^{\prime}$, 
such as in $\qty[ \bm{S}_{k_1} ]_{i i^{\prime}}^{L}$,
indicates the interaction between $J_{i, \mathrm{row}}^{L}$ and $J_{i, \mathrm{col}}^{\prime L}$ for a fixed $i \in \{ 1, 2, \ldots, p \}$.
The local interactions in Table \ref{tb:low-rank} are factorized using column-pivoted QR factorization
to obtain the pair of low-rank coefficients $R_{i}^{j, L}$ and indices $\tilde{J}_{i, \mathrm{col}}^{j, L}$,
and the pair of low-rank coefficients $L_{i}^{j, L}$ and indices $\tilde{J}_{i, \mathrm{row}}^{j, L}$ \cite{cheng2005compression}, for $j = 1, 2, 3$.
The indices $\tilde{J}_{i, \mathrm{row}}^{j, L}$ and $\tilde{J}_{i, \mathrm{col}}^{j, L}$, $j = 1, 2, 3$, are obtained as the first $k$ pivots in column-pivoted QR factorization.

\begin{table}[h]
  \caption{Target matrix for column-pivoted QR factorization in discretized direct-indirect Burton--Miller equation for obtaining coefficient matrices $R_{i}^{j, L}$ and $L_{i}^{j, L}$, and 
    skeleton indices $\tilde{J}_{i, \mathrm{col}}^{L} = \qty{ \tilde{J}_{i, \mathrm{col}}^{1, L},  \tilde{J}_{i, \mathrm{col}}^{2, L}, \tilde{J}_{i, \mathrm{col}}^{3, L}}$ and $\tilde{J}_{i, \mathrm{row}}^{L} = \qty{ \tilde{J}_{i, \mathrm{row}}^{1, L},  \tilde{J}_{i, \mathrm{row}}^{2, L}, \tilde{J}_{i, \mathrm{row}}^{3, L}}$,
    for $i = 1, 2, \ldots, p$ and for $j = 1, 2, 3$;
    the superscript $H$ indicates the complex conjugate transpose
  }
  \label{tb:low-rank}%
  \begin{tabular}{llll}
    \toprule
    Resulting matrix & Resulting skeleton indices & Target for column-pivoted QR factorization \\
    \midrule
    $R_{i}^{1, L}$ & $\tilde{J}_{i, \mathrm{col}}^{1, L}$ & $\qty[ \bm{D}_{k_0} ]_{i^{\prime} i}^{L} + \beta \qty[ \bm{N}_{k_0} ]_{i^{\prime} i}^{L}$ & $(\in \mathbb{C}^{n^{\prime} \times n})$ \\
    $R_{i}^{2, L}$ & $\tilde{J}_{i, \mathrm{col}}^{2, L}$& $\qty[ \bm{S}_{k_0}]_{i^{\prime} i}^{L} + \beta [\bm{D}_{k_0}^{\bm *} ]_{i^{\prime} i}^{L}$ & $(\in \mathbb{C}^{n^{\prime} \times n})$ \\
    $R_{i}^{3, L}$ & $\tilde{J}_{i, \mathrm{col}}^{3, L}$ & $\qty[\bm{S}_{k_1}]_{i^{\prime} i}^{L}$ & $(\in \mathbb{C}^{n^{\prime} \times n})$ \\
    $L_{i}^{1, L}$ & $\tilde{J}_{i, \mathrm{row}}^{1, L}$ & $\qty( [\bm{S}_{k_1}]_{i i^{\prime}}^{L})^{H}$ & $(\in \mathbb{C}^{n^{\prime} \times n})$ \\
    $L_{i}^{2, L}$ & $\tilde{J}_{i, \mathrm{row}}^{2, L}$ & $\qty( [\bm{D^*}_{k_1}]_{i i^{\prime}}^{L})^{H}$ & $(\in \mathbb{C}^{n^{\prime} \times n})$ \\
    $L_{i}^{3, L}$ & $\tilde{J}_{i, \mathrm{row}}^{3, L}$ & $ \mqty( \qty[ \bm{D}_{k_0} ]_{i i^{\prime}}^{L} + \beta \qty[ \bm{N}_{k_0} ]_{i i^{\prime}}^{L} & \quad -\varepsilon_0 \left\{ \qty[\bm{S}_{k_0}]_{i i^{\prime}}^{L} + \beta \qty[\bm{D^*}_{k_0}]_{i i^{\prime}}^{L} \right\})^{H} $ & $(\in \mathbb{C}^{2n^{\prime} \times n})$ \\
    \botrule
  \end{tabular}
\end{table}

In this low-rank approximation, the discretized direct-indirect mixed Burton--Miller equation may be faster than
the discretized ordinary Burton--Miller equation considering the number of required floating-point operations.
It is known that QR factorization based on the Householder reflection for a matrix of size $p \times q$ requires
approximately $2 p q^2 - \frac{2}{3} q^3$ floating-point operations.
From the size of the target matrices in Table \ref{tb:low-rank},
we can see that the discretized direct-indirect mixed Burton--Miller equation needs
five matrices of size $n^{\prime} \times n$ and a matrix of size $2n^{\prime} \times n$,
while the discretized ordinary Burton--Miller equation needs
four matrices of size $2n^{\prime} \times n$ \cite{Matsumoto2019burton}.
Thus,
the number of floating-point operations required to compute $L_{i}^L$ and $R_{i}^L$ for each $i$ can be estimated approximately as
\begin{align}
  &14 n^{\prime} n^2 - \frac{12}{3} n^3 &\, &\text{for the discretized direct-indirect mixed Burton--Miller equation, and} \\
  &16 n^{\prime} n^2 - \frac{8}{3} n^3 &\, &\text{for the discretized ordinary Burton--Miller equation}.
\end{align}
From this estimate, we see that the difference of the number of floating-point operations is $2 n^{\prime} n^2 + (4/3) n^3$, which is the speedup factor of the discretized direct-indirect mixed Burton--Miller equation against to the ordinary formulation.
It is confirmed in Section \ref{sec:numerical_examples} that the former formulation is indeed faster.

The multi-level low-rank approximation using the proxy method can be performed at upper levels in a similar way to that at the leaf level.
The details of the multi-level algorithm are described in \cite{MATSUMOTO2025106148}.
We use the multi-level algorithm in the numerical examples.

\subsubsection{Compression technique for linear equations}
We can compress \eqref{eq:dis_mixedBM} to a system of linear equations with a small number of degrees of freedom using a compression technique based on the shared coefficient weak admissibility condition discussed in the previous subsection.
As the fast direct solver without this compression scheme is the same as the method in \cite{MATSUMOTO2025106148},
we omit its explanation in this article.
The off-diagonal low-rank approximated blocked linear equations given by
\begin{align}
  &\mqty[
  A_{11}^{L} & L_{1}^L S_{12}^L R_{2}^L & \cdots & L_{1}^L S_{1p}^L R_{p}^L \\
  L_{2}^L S_{21}^L R_{1}^L & A_{22}^{L} & \cdots & L_{2}^L S_{2p}^L R_{p}^L \\
  \vdots & \vdots &\ddots &\vdots \\
  L_{p}^L S_{p1}^L R_{1}^L & L_{p}^L S_{p2}^L R_{2}^L & \cdots & A_{pp}^{L}
  ]
  \mqty[
  x_1^{L} \\
  x_2^{L} \\
  \vdots \\
  x_p^{L}
  ]
  =
  \mqty[
  f_1^{L} \\
  f_2^{L} \\
  \vdots \\
  f_p^{L}
  ],
\end{align}
whose degrees of freedom is $3np = 3N$,
are compressed to 
\begin{multline}
  \mqty[
    \qty{R_1^{L} (A_{11}^{L})^{-1} L_1^{L}}^{-1} & S_{12}^{L} & \cdots & S_{1p}^{L} \\
    S_{21}^{L} & \qty{R_2^{L} (A_{22}^{L})^{-1} L_2^{L}}^{-1} & \cdots & S_{2p}^{L} \\
    \vdots & \vdots &\ddots &\vdots \\
    S_{p1}^{L} & S_{p2}^{L} & \cdots & \qty{R_p^{L} (A_{pp}^{L})^{-1} L_p^{L}}^{-1}
  ]
  \mqty[
    R_1^{L} x_1^{L} \\
    R_2^{L} x_2^{L} \\
    \vdots \\
    R_p^{L} x_p^{L}
  ]
  \\ =
  \mqty[
    \qty{R_1^{L} (A_{11}^{L})^{-1} L_1^{L}}^{-1} R_{1}^{L}(A_{11}^{L})^{-1}f_1^{L} \\
    \qty{R_2^{L} (A_{22}^{L})^{-1} L_2^{L}}^{-1} R_{2}^{L}(A_{22}^{L})^{-1}f_2^{L} \\
    \vdots \\
    \qty{R_p^{L} (A_{pp}^{L})^{-1} L_p^{L}}^{-1} R_{p}^{L}(A_{pp}^{L})^{-1}f_p^{L}
  ],
  \label{eq:compressed}
\end{multline}
whose degrees of freedom is $3kp$, with  $k \ll n$.
In this compression, the efficient LU factorization in Theorem \ref{thm:LU} and Corollary \ref{lem:block_LU}
can be applied to compute $(A_{ii}^{L})^{-1} f_{i}^{L}$ and $(A_{i}^{L})^{-1} L_{i}^{L}$, respectively, for $i = 1, 2, \ldots, p$.

\begin{remark}
  Note that this efficient LU factorization can be performed only at the leaf level since the diagonal matrix $A_{ii}^{\ell}$ at upper levels $\ell$, $\ell = L - 1, L -2, \ldots, 0$, does not have an appropriate sub-block alignment to be LU factorized as done in Theorem \ref{thm:LU}.
\end{remark}

\section{Numerical examples} \label{sec:numerical_examples}
This section demonstrates the performance of the proposed formulation, namely the direct-indirect mixed Burton--Miller equation \eqref{eq:mixed_BM_equation}.
The direct solvers based on Theorem \ref{thm:LU} and Corollaries \ref{lem:block_LU} and \ref{lem:for_ssm_RHS}
are used in Sections \ref{sec:LU_numerical}, \ref{sec:fastLU}, and \ref{sec:ssm}, respectively.
The numerical examples were calculated using a single node of TSUBAME4.0,
which has two 96-core AMD EPYC 9654 CPUs (total number of cores: 192;
maximum frequency: 3708 MHz) and 768 GB of 0.92-TB/s DDR5 memory, at the Institute of Science Tokyo.
The numerical computation program was implemented using \texttt{C++} and the linear algebra package \texttt{Eigen3} \cite{eigenweb}.
The implementation was compiled using \texttt{g++} and parallelized using \texttt{OpenMP}.
In the program, the Hankel functions are calculated using the fixed version of \texttt{SLATEC} to be thread-safe \cite{lloda_slatec-bessel-cpp_2024}.

\begin{remark}
  The problems discussed in Sections \ref{sec:LU_numerical}, \ref{sec:fastLU}, and \ref{sec:ssm}
  differ in the granularity of their parallelization units.
  In these numerical examples, the for-loop is executed in parallel using \texttt{OpenMP}.
  When there is a nested for-loop,
  it is generally more efficient to parallelize the outer for-loop.
  In Section \ref{sec:LU_numerical}, the construction of the coefficient matrix is parallelized.
  In Section \ref{sec:fastLU}, the compression of the diagonal block of the coefficient matrix is parallelized with respect to the cells of the tree.
  In Section \ref{sec:ssm}, the step used to find a solution for the linear equations on each integration point on the path of the integral of the Sakurai--Sugiura projection method \cite{asakura2009numerical} is parallelized.
  Therefore, the scaling efficiency varies among the numerical examples.
\end{remark}

\subsection{Direct solver based on LU factorization} \label{sec:LU_numerical}
Using Theorem \ref{thm:LU}, we expect to solve
the discretized direct-indirect mixed Burton--Miller equation faster than
the discretized ordinary Burton--Miller equation.
The parameters used for solving Helmholtz transmission problems via boundary integral equations are listed in Table \ref{tb:lu_parameters}.
Recall that $N$ is the number of the quadrature points of the Nystr\"om method,
so that the degrees of freedom are $3N$ and $2N$ in the mixed and ordinary formulations, respectively.
\begin{table}[h]
  \caption{
    Parameters used in Section \ref{sec:LU_numerical}
  }
  \label{tb:lu_parameters}
  \begin{tabular}{ll}
    \toprule
    Parameter & Value \\
    \midrule
    Scatterer shape & Circle with radius 1.0 \\
    $\varepsilon_0$ & 1 \\
    Pair $\{ \omega, \varepsilon_1 \}$ & \{1, 2\}, \{2, 5\}, \{10, 10\} \\
    $N$ & $100 \times 2^{i}$ for $i = 1, 2, \ldots, 8$ \\
    Incident wave & Plane wave $e^{i k_0 x_1}$ \\
    \botrule
  \end{tabular}
\end{table}

\begin{figure}[htb]
  \centering
  \includegraphics[width=1.0\linewidth]{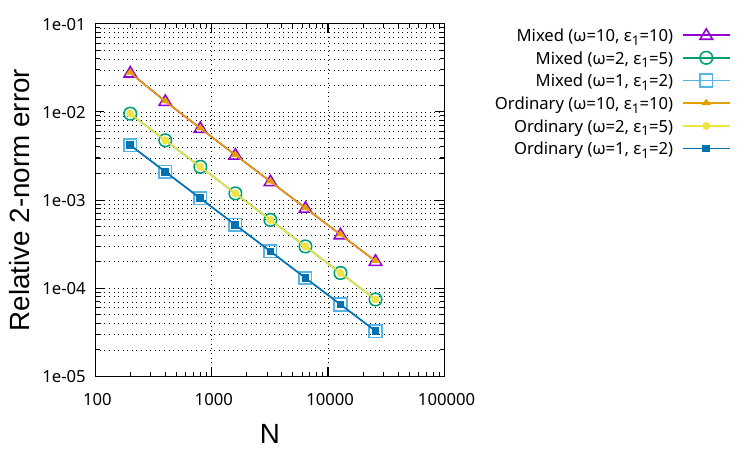}
  \caption{
    Relative 2-norm errors of numerical solution of each formulation with respect to analytical solution obtained using direct solver based on LU factorization;
    the figure presents two formulations for each pair of frequency ($\omega$) and internal material parameter ($\varepsilon_1$), namely the discretized direct-indirect mixed Burton--Miller equation (``Mixed'') and the discretized ordinary Burton--Miller equation (``Ordinary'')
  }
  \label{fig:convLU_error}
\end{figure}
In this setting, we can calculate the analytical solution of corresponding Helmholtz transmission problems.
We first compare the obtained numerical solution for $\bm u$ and $\bm q$ to the analytical solution on $\Gamma$.
Let $u_{\mathrm{ana}}$ be the analytical solution vector on the quadrature points of the Nystr\"om discretization.
Let $q_{\mathrm{ana}}$ be the normal derivative of $u_{\mathrm{ana}}$.
We evaluate the relative 2-norm error as
\begin{equation}
  \frac{\norm{\qty[\bm{u}, \bm{q}]^T - \qty[u_{\mathrm{ana}}, q_{\mathrm{ana}}]^T}_2}{\norm{\qty[u_{\mathrm{ana}}, q_{\mathrm{ana}}]^T}_2},
  \label{eq:norm}
\end{equation}
where $\norm{\cdot}_2$ is the Euclidean norm.
Fig. \ref{fig:convLU_error} shows the relative 2-norm error with respect to the analytical solution on the quadrature points.
We can see that the numerical solutions obtained from both formulations well match the analytical solutions.
Furthermore, they converge to the analytical solution at a rate of $O(N^{-1})$, which is the same convergence rate found in an error analysis of the zeta-corrected quadrature \cite{wu2023unified}.

Fig. \ref{fig:convLU_time} shows the computational time versus $N$ for various values of $\omega$ and $\varepsilon_1$.
As shown, the discretized direct-indirect mixed Burton--Miller equation solves the problem faster than does
the discretized ordinary Burton--Miller equation in most cases.
The results in Fig. \ref{fig:convLU_time} are summarized in Table \ref{tb:speedup}.
Note that this table also includes the results for the fast direct solver, which are discussed later.
From Table \ref{tb:speedup}, the proposed formulation, namely the direct-indirect mixed Burton--Miller equation, is more efficient than the ordinary equation
despite the proposed formulation having a $1.5$ times greater number of degrees of freedom.
The expected speedup for the solver was about 12.5\% (see Section \ref{sec:LU}),
but the actual speedup was about 40\% over the entire numerical process.
Note that the speedup for the given code is machine-dependent.
\begin{figure}[htb]
  \centering
  \begin{minipage}[t]{0.49\textwidth}
    \centering
    \includegraphics[width=1.0\linewidth]{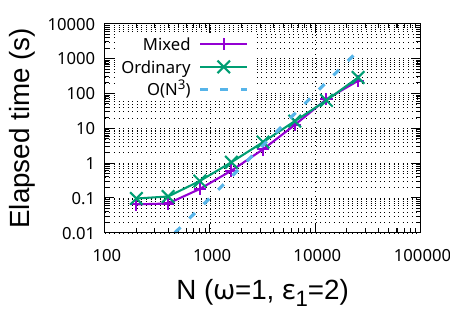}
  \end{minipage}
  \hfill
  \begin{minipage}[t]{0.49\textwidth}
    \centering
    \includegraphics[width=1.0\linewidth]{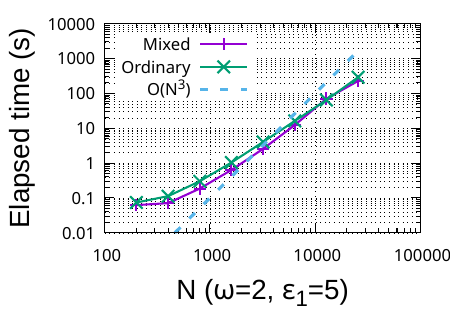}
  \end{minipage}
  \includegraphics[width=0.5\linewidth]{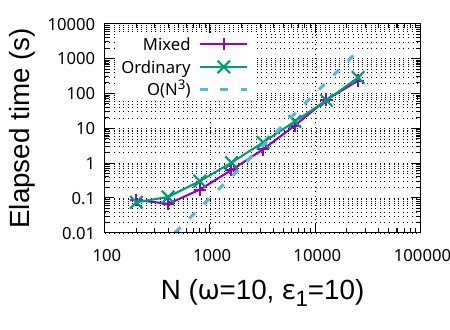}
  \caption{
    Computational time versus $N$ obtained using LU factorization for solving discretized direct-indirect mixed Burton--Miller equation (``Mixed'') and
    discretized ordinary Burton--Miller equation (``Ordinary'') in each pair of frequency $\omega$ and material parameter inside $\varepsilon_1$;
    only the mixed equation benefits from Theorem \ref{thm:LU}
  }
  \label{fig:convLU_time}
\end{figure}

\begin{table}[h]
  \caption{
    Average speedup of computational time of discretized direct-indirect mixed Burton--Miller equation $(\mathrm{Time}_{\mathrm{mixed}})$ compared to discretized ordinary Burton--Miller equation $(\mathrm{Time}_{\mathrm{ordinary}})$ for direct solvers in Sections \ref{sec:LU_numerical} and \ref{sec:fastLU}
  }
  \label{tb:speedup}
  \begin{tabular}{llll}
    \toprule
    Solver & $\omega$ & $\varepsilon_1$ & Average speedup $\qty(\frac{\mathrm{Time}_{\mathrm{ordinary}}}{\mathrm{Time}_{\mathrm{mixed}}})$ \\
    \midrule
    LU factorization & 1 & 2 & 1.444 \\
    & 2 & 5 & 1.390 \\
    & 10 & 10 & 1.383 \\
    Fast direct solver & 1 & 2 & 1.013 \\
    & 2 & 5 & 1.053 \\
    & 10 & 10 & 1.013 \\
    \botrule
  \end{tabular}
\end{table}

\subsection{Fast direct solver} \label{sec:fastLU}
We next discuss the performance for the fast direct solver described in Section \ref{sec:fds}.
For comparison, we use the fast direct solver for the discretized ordinary Burton--Miller equation \eqref{eq:dis_BM},
which is the same as the method in \cite{Matsumoto2019burton}
if the periodic Green's function is replaced by the fundamental solution.
In the multi-level algorithm of the fast direct solver, 
the number of skeletons increases with each level in the binary tree
by $1.15$ times that for the previous level in this numerical example.
For example, if the number of skeletons at the leaf level is 30,
it increases to 34 at the level above the leaf level,
and then to 39 at the level above that.
Table \ref{tb:fast_lu_parameters} shows the parameters used in Section \ref{sec:fastLU}.
In the table, the number of initial skeletons refers to the number of skeletons at the leaf level of the tree.
The material parameters and scatterer shape are the same as those in the previous section.
\begin{table}[h]
  \caption{
    Parameters used in Section \ref{sec:fastLU}
  }
  \label{tb:fast_lu_parameters}
  \begin{tabular}{ll}
    \toprule
    Parameter & Value \\
    \midrule
    Scatterer shape & Circle with radius 1.0 \\
    $\varepsilon_0$ & 1 \\
    Pair $\{ \omega, \varepsilon_1 \}$ & \{1, 2\}, \{2, 5\}, \{10, 10\} \\
    $N$ & $100 \times 2^{i}$ for $i = 10, 11, \ldots, 15$ \\
    Incident wave & Plane wave $e^{i k_0 x_1}$ \\
    Number of initial skeletons & 40 \\
    \botrule
  \end{tabular}
\end{table}

Fig. \ref{fig:fastLU_error} shows the relative 2-norm error evaluated using \eqref{eq:norm} with respect to $N$.
We can see that the numerical solutions obtained from both formulations well match the analytical solutions even when the fast direct solver is used.
Similar to the result in the previous section, the formulations almost converge to the analytical solution at a rate of $O(N^{-1})$.
It is conjectured that the slight deviation of relative errors from the expected $O(N^{-1})$ behavior at $N = 100 \times 2^{15} = 3{,}276{,}800$ is caused by numerical error in calculating the Hankel function with a small argument.
\begin{figure}[htb]
  \centering
  \includegraphics[width=1.0\linewidth]{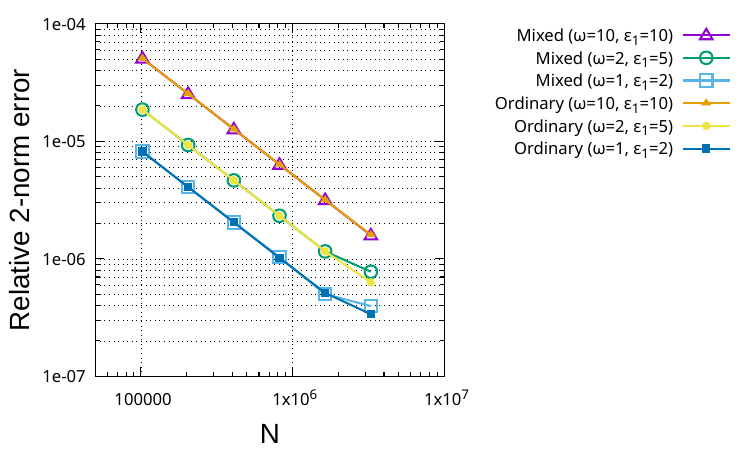}
  \caption{
    Relative 2-norm errors of numerical solution of each formulation with respect to analytical solution obtained using fast direct solver;
    the figure presents two formulations for each pair of frequency ($\omega$) and internal material parameter ($\varepsilon_1$), namely the discretized direct-indirect mixed Burton--Miller equation (``Mixed'') and the discretized ordinary Burton--Miller equation (``Ordinary'')
  }
  \label{fig:fastLU_error}
\end{figure}

Fig. \ref{fig:fastLU_time} shows the computational time versus $N$ for various values of $\omega$ and $\varepsilon_1$.
As shown, the discretized direct-indirect mixed Burton--Miller equation solves the problem slightly faster than does
the discretized ordinary Burton--Miller equation.
The results in Fig. \ref{fig:fastLU_time} are summarized in Table \ref{tb:speedup}.
The fast direct solver based on the proposed formulation achieved only a slight speedup compared to the direct solver based on LU factorization.
Nevertheless, the proposed mixed formulation has advantages since it can solve the problem faster with comparable accuracy despite the increased degrees of freedom of the linear equations.
\begin{figure}[htb]
  \centering
  \begin{minipage}[t]{0.49\textwidth}
    \centering
    \includegraphics[width=1.0\linewidth]{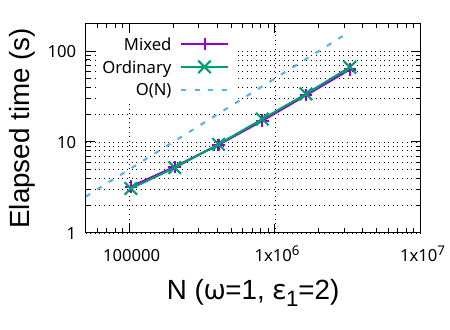}
  \end{minipage}
  \hfill
  \begin{minipage}[t]{0.49\textwidth}
    \centering
    \includegraphics[width=1.0\linewidth]{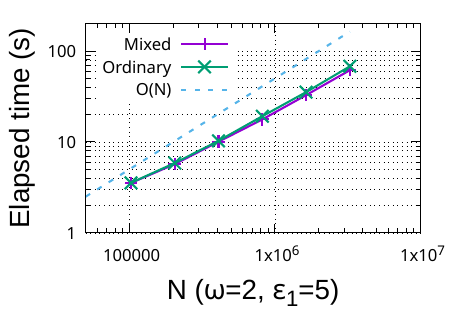}
  \end{minipage}
  \includegraphics[width=0.5\linewidth]{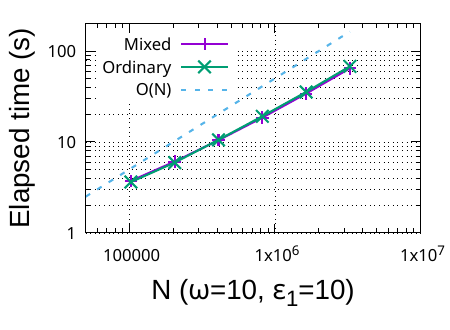}
  \caption{
    Computational time versus $N$ obtained using fast direct solver in Section \ref{sec:fds} for solving discretized direct-indirect mixed Burton--Miller equation (``Mixed'') and
    discretized ordinary Burton--Miller equation (``Ordinary'') in each pair of frequency $\omega$ and material parameter inside $\varepsilon_1$
  }
  \label{fig:fastLU_time}
\end{figure}

\subsection{Calculation of eigenvalues of boundary integral equations} \label{sec:ssm}
We now demonstrate a numerical example using Corollary \ref{lem:for_ssm_RHS}.
It is known that the Helmholtz transmission problem \eqref{eq:problem} may have
a non-trivial solution when we allow a complex $\omega$, provided that $\omega$ is not a negative real number.
We refer to the eigenvalues as a set of $\omega$ for which the Helmholtz transmission problem \eqref{eq:problem} is not uniquely solvable.
Moreover, boundary integral equations, including \eqref{eq:mixed_BM_equation} and \eqref{eq:ordi_BM_equation},
have not only the eigenvalues of the Helmholtz transmission problem but also additional eigenvalues.
Therefore, we also refer to the eigenvalues as a set of $\omega$ for which the boundary integral equations \eqref{eq:mixed_BM_equation} and \eqref{eq:ordi_BM_equation} are not uniquely solvable.
Since $\omega$ affects the boundary integral operator nonlinearly,
these eigenvalues can only be found by solving a nonlinear eigenvalue problem.
We use the Sakurai--Sugiura projection method \cite{asakura2009numerical} to solve this nonlinear eigenvalue problem by converting it to a generalized eigenvalue problem and then finding the eigenvalues of a region bounded by a closed curve in the complex plane.
In the Sakurai--Sugiura projection method,
we repeatedly solve the discretized direct-indirect Burton--Miller equation \eqref{eq:dis_mixedBM} or
the discretized ordinary Burton--Miller equation \eqref{eq:dis_BM}.
These solutions are required for frequencies $\omega$ on a closed curve in the complex plane, using randomly generated right-hand side vectors.
Thus, we can use Corollary \ref{lem:for_ssm_RHS} for the discretized direct-indirect Burton--Miller equation.

For this numerical example, we use a circular scatterer.
From an analysis for a circular scatterer with radius $a$ \cite{matsumoto2025injectivity},
the eigenvalues of both \eqref{eq:mixed_BM_equation} and \eqref{eq:ordi_BM_equation} are the same;
they are given as the zeros with respect to $\omega$ of the following two relations:
\begin{align}
  -\varepsilon_{0} k_{1} H_{n}^{(1)}(k_{0} a_{}) {J_{n}}^{\prime}(k_{1} a_{}) + \varepsilon_{1} k_{0} {H_{n}^{(1)}}^{\prime}(k_{0} a_{}) {J_{n}}^{}(k_{1} a_{}) = 0, \\
  H_{n}^{(1)}(k_{1} a_{}) \qty(J_{n}(k_{0} a_{}) + \alpha k_{0} {J_{n}}^{\prime}(k_{0} a_{})) = 0,
\end{align}
for some integer $n \in \mathbb{Z}$.
Note that $k_j = \omega \sqrt{\varepsilon_j}$ for $j = 0, 1$.
The left-hand side of the second equation above vanishes if $\omega$ satisfies $H_{n}^{(1)}(k_{1} a_{}) = 0$.
Some of the zeros of the Hankel function have been tabulated \cite{doring1966complex}.
Therefore, we verify whether the Sakurai--Sugiura projection method for both formulations
can determine the zeros of the Hankel function.
In this verification, we use the direct solver used in Section \ref{sec:LU_numerical} considering the results in Table \ref{tb:speedup}.
We used the lowest-order zeta-corrected quadrature in Section \ref{sec:LU_numerical} in view of the extension to the fast direct solver that requires the weak admissibility condition.
Note that the use of Theorem \ref{thm:LU} is not restricted to the lowest order of the zeta-corrected quadrature.
We use a high-order zeta-corrected quadrature, whose order is $31$, in this numerical example \cite{wu2023unified}.
The parameters used in the Sakurai--Sugiura projection method are listed in Table \ref{tb:ssm_parameters}.
In this table, the parameters are selected such that both formulations of the boundary integral equations have zeros of $H_{n}^{(1)}(\omega)$ as eigenvalues.
\begin{table}[h]
  \caption{
    Parameters used in Sakurai--Sugiura projection method (SSM)
  }
  \label{tb:ssm_parameters}
  \begin{tabular}{ll}
    \toprule
    Parameter & Value \\
    \midrule
    Integral path of SSM & Square with side length 0.1 centered on target eigenvalue \\
    Number of integration points of SSM & 48 per side \\
    Number of moments of SSM & 4 \\
    Number of right-hand sides of blocked SSM & 4 \\
    Scatterer shape & Circle with radius 0.5 \\
    $\varepsilon_0$ & 1 \\
    $\varepsilon_1$ & 4 \\
    $N$ & 4000 \\
    \botrule
  \end{tabular}
\end{table}

The verification results are shown in Table \ref{tb:ssm_eigenvalues}.
We can see that the mixed formulation can compute almost the same eigenvalues as those computed by the ordinary formulation.
These computed eigenvalues well match the reference ones.
Furthermore, computation using the mixed formulation is approximately 10\% faster than that using the ordinary formulation.
\begin{table}[h]
  \caption{
    Part of eigenvalues of boundary integral equations \eqref{eq:mixed_BM_equation} (``Mixed'')
    and \eqref{eq:ordi_BM_equation} (``Ordinary'') computed using Sakurai--Sugiura projection method;
    the eigenvalue of the number 1 corresponds to the zero of $H_{2}^{(1)}(z)$ and $H_{-2}^{(1)}(z)$ and
    the eigenvalue of the number 2 corresponds to the zero of $H_{3}^{(1)}(z)$ and $H_{-3}^{(1)}(z)$;
    the reference value \cite{doring1966complex} was computed only for $H_{i}^{(1)}(z)$, for a positive integer $i$, since the zeros of $H_{i}^{(1)}(z)$ and $H_{-i}^{(1)}(z)$ are the same
  }
  \label{tb:ssm_eigenvalues}
  \begin{tabular}{lllll}
    \toprule
    No. &Formulation & Time (s) & Real part of eigenvalue & Imaginary part of eigenvalue \\
    \midrule
    1 & Mixed \eqref{eq:dis_mixedBM} & 98.895 & 0.42948496520839791   &    $-$1.2813737976521153 \\
    & & & 0.42948496521037804   &    $-$1.2813737976570456 \\
    & Ordinary \eqref{eq:dis_BM} & 115.910 &   0.42948496520023421   &    $-$1.2813737976543407 \\
    & & & 0.42948496521115348   &    $-$1.2813737976568351 \\
    & (Reference, Doring \cite{doring1966complex}) & -- & $0.4294849652$  & $-1.2813737977$ \\
    \hline
    2 & Mixed \eqref{eq:dis_mixedBM} & 114.206 & 1.3080120322683002    &   $-$1.6817888047453846  \\ 
    & & & 1.3080120323125690   &    $-$1.6817888047547371 \\
    & Ordinary \eqref{eq:dis_BM} & 125.449 & 1.3080120322740318   &    $-$1.6817888047448746 \\
    & & & 1.3080120322880620   &    $-$1.6817888047301592 \\
    & (Reference, Doring \cite{doring1966complex}) & -- & $1.3080120323$  & $-1.6817888047$ \\
    \botrule
  \end{tabular}
\end{table}

\section{Concluding remarks} \label{sec:conclusion}
This paper proposed a direct-indirect mixed Burton--Miller boundary integral equation
for solving Helmholtz transmission problems.
This formulation has three unknowns, one more than the number of unknowns for the ordinary formulation.
However, it was demonstrated by numerical examples that a direct solver based on the discretized direct-indirect Burton--Miller boundary integral equation is faster than that based on the ordinary formulation
owing to the sparse arrangement of integral operators.
The techniques used for speeding up computations were explained in detail and
the well-posedness of the proposed boundary integral equation was proven in the framework of the mapping property of boundary integral operators in H\"older space.

The proposed mixed formulation has some limitations.
It has a larger number of degrees of freedom than that of the ordinary formulation and therefore requires a greater amount of memory.
This means that for a given computer system, the size of the problem that can be solved is slightly smaller.
To address this issue, we considered applying the proposed mixed formulation to the fast direct solver.
This problem occurs since the implementation remembers the zero blocks of the coefficient matrix.
If we modify the implementation so that it does not remember these zero blocks,
this issue should be mostly solved.

In future work, to speed up computations, it will be worthwhile to investigate whether the proposed formulation remains competitive with existing formulations when high-performance computing techniques, such as distributed memory parallelization and offloading to accelerators (e.g., graphics processing units), are fully utilized.
Although the mathematical analysis would become more difficult, it may be possible to extend the proposed efficient direct solvers to transmission problems governed by other partial differential equations, such as Maxwell's equations or the Navier--Cauchy equation.
In this paper, the proposed boundary integral equation is well-posed if the boundary is of class $C^2$.
The well-posedness of the proposed boundary integral equation is expected to be extended to the Lipschitz boundary using the framework in Sobolev space.

\bmhead{Acknowledgements}
This work was supported by
the Japan Society for the Promotion of Science under KAKENHI grant numbers 23H03413, 23H03798, 24K17191, and 24K20783.
This study made use of the computational resources of TSUBAME4.0 at the Institute of Science Tokyo
provided through the projects 
``Joint Usage/Research Center for Interdisciplinary Large-scale Information Infrastructures (JHPCN)''
and ``High Performance Computing Infrastructure (HPCI)'' in Japan (project ID jh250045).


\end{document}